\documentclass[11pt,a4paper]{article}

\usepackage{amsmath}
\usepackage{amsfonts}
\usepackage{amssymb}
\usepackage{centernot}
\usepackage[margin=3cm]{geometry}
\usepackage{enumitem}
\usepackage{hhline}
\usepackage{graphicx,fancybox}
\usepackage{tikz}
\usepackage{subcaption}
\usepackage{amsthm}
\usepackage{pgf}
\usepackage[font={small}]{caption}
\usepackage{etoolbox}
\usepackage{array, makecell, cellspace}
\usepackage{soul}

\DeclareMathOperator*{\argmin}{argmin}
\DeclareMathAlphabet\mbc{OMS}{cmsy}{b}{n}

\setlist[enumerate]{label={(\roman*)},topsep=3pt}

\newcommand{\gra}{\operatorname{gra}}
\newcommand{\zer}{\operatorname{zer}}
\newcommand{\ran}{\operatorname{ran}}
\newcommand{\dom}{\operatorname{dom}}
\newcommand{\Id}{\operatorname{Id}}
\newcommand{\prox}{\operatorname{prox}}

\newcommand{\ve}{\operatorname{vec}}
\newcommand{\bs}{\boldsymbol}
\newcommand{\Hi}{\mathcal{H}}
\newcommand{\R}{\mathbb{R}}
\newcommand{\C}{\mathcal{C}}

\def\tto{\rightrightarrows}
\def\wto{\rightharpoonup}

\usepackage{titlesec}
\usepackage{titling}
\usepackage{anyfontsize}
\pretitle{\vspace{-1cm}\begin{center}\fontsize{16}{20}\bfseries\sffamily}
\pretocmd{\theequation}{\small}{}{}

\makeatletter
\def\th@plain{%
	\thm@notefont{}
	\itshape 
}
\def\th@definition{%
	\thm@notefont{}
	\normalfont 
}

\usepackage[colorlinks=true,linkcolor=blue,citecolor=blue,urlcolor=magenta]{hyperref}
\usepackage[capitalise]{cleveref}

\newtheorem{theorem}{Theorem}[section]
\newtheorem{proposition}[theorem]{Proposition}

\newtheorem{lemma}[theorem]{Lemma}
\newtheorem{fact}[theorem]{Fact}

\theoremstyle{definition}
\newtheorem{definition}[theorem]{Definition}

\newtheorem{remark}[theorem]{Remark}
\newtheorem{example}[theorem]{Example}

\Crefname{fact}{Fact}{Facts}
\Crefname{assumption}{Assumption}{Assumptions}
\Crefname{enumi}{}{}

\numberwithin{equation}{section}

\def\fw{0.99}

\title{A Product Space Reformulation with Reduced Dimension for Splitting Algorithms}

\author{Rub\'en Campoy \thanks{Department of Statistics and Operational Research, Universitat de Val\`encia, Valencia, \textsc{Spain}. E-mail:~\href{mailto:ruben.campoy@uv.es}{ruben.campoy@uv.es}}}

\begin{document}
\maketitle

\begin{abstract}
In this paper we propose a product space reformulation to transform monotone inclusions described by finitely many operators on a Hilbert space into equivalent two-operator problems. Our approach relies on Pierra's classical reformulation with a different decomposition, which results on a reduction of the dimension of the outcoming product Hilbert space. We discuss the case of not necessarily convex feasibility and best approximation problems. By applying existing splitting methods to the proposed reformulation we obtain new parallel variants of them with a reduction in the number of variables. The convergence of the new algorithms is straightforwardly derived with no further assumptions. The computational advantage is illustrated through some numerical experiments.

\paragraph{\small Keywords} Pierra's product space reformulation $\cdot$ Splitting algorithm $\cdot$ Douglas--Rachford algorithm $\cdot$ Monotone inclusions $\cdot$ Feasibility problem $\cdot$ Projection methods
\paragraph{\small MSC\,2020:} 47H05 $\cdot$ 47J25 $\cdot$ 49M27 $\cdot$ 65K10 $\cdot$ 90C30
\end{abstract}

\section{Introduction}

A problem of great interest in optimization and variational analysis is the monotone inclusion consisting in finding a zero of a monotone operator. In many practical applications, such operator can be decomposed as a sum of finitely many maximally monotone operators. The problem takes then the form
\begin{equation}\label{prob:zerosum}
\text{Find } x\in\Hi \text{ such that } 0 \in A_1(x)+A_2(x)+\cdots+A_r(x),
\end{equation}
where $\Hi$ is a Hilbert space and $A_1,A_2,\ldots,A_r:\Hi\tto\Hi$ are maximally monotone. When the sum is itself maximally monotone, in theory, inclusion \eqref{prob:zerosum} could be numerically solved by the well-known \emph{proximal point algorithm} \cite{RockaProx}. However, this method requires the computation of the resolvent of the whole operator at each iteration, which is not usually available. In fact, computing the resolvent of a sum at a given point $q\in\Hi$, i.e.,
\begin{equation}\label{prob:resolvent}
\text{Find } p\in J_{\sum_{i=1}^r A_i}(q),
\end{equation}
where $J_{A}$ denotes the {resolvent} of an operator $A$, is a problem of interest itself which arises in some optimization subroutines as well as in direct applications such as best approximation, image denoising and partial differential equations (see, e.g.,~\cite{aragon2020strengh}).

Splitting algorithms take advantage of the decomposition and activate each operator separately, either by direct evaluation (forward steps) or via its resolvent (backward steps), to construct a sequence that converges to a solution of the problem. Splitting algorithms include, in particular, the so-called \emph{projection methods}, which permit to find a point (or the closest point) in the intersection of a collection of sets by computing individual projections onto them. Classical splitting algorithms for monotone inclusions include the \emph{Forward-Backward algorithm} and its variants, see, e.g.,~\cite{BC17,cevher2019reflected,malitsky2018forward,tseng2000modified}, and the \emph{Douglas--Rachford algorithm}~\cite{DR56,LM79}, among others (see, e.g., \cite[Chapter~23]{BC17}). On the other hand, different splitting algorithms for computing the resolvent of a sum can be found in, e.g,~\cite{adly2019decomposition,aragon2019computing,combettes2009iterative,dao2019resolvent}. See also the recent unifying framework \cite{aragon2020strengh}.

Most splitting algorithms in the literature are devised for a sum of two operators, whereas there exist just a few three-operator extensions, see, e.g.,~\cite{davis2017three,rieger2020backward,ryu}. In general, problems \eqref{prob:zerosum}--\eqref{prob:resolvent} are tackled by splitting algorithms after applying \emph{Pierra's product space reformulation}~\cite{Pierrathesis,Pierra}. This technique constructs an equivalent two-operator problem, embedded in a product Hilbert space, that preserves computational tractability in the sense that the resolvents of the new operators can be readily computed. However, since each operator in the original problem requires one dimension in the product space, this technique may result numerically inefficient when the number of operators is too large.

In this work we propose an alternative reformulation, based on Pierra's classical one, which reduces the dimension of the resulting product Hilbert space. Our approach consists in merging one of the operators with the normal cone to the diagonal set, what allows to remove one dimension in the product space. In fact, this seems a more natural embedding than Pierra's one since it reproduces exactly the original problem when this is initially defined by two operators (see~\Cref{rem:recovers}). We would like to note that this reformulation has already been used in other frameworks. For instance, it was employed in \cite{krugerPS} for deriving necessary conditions for extreme points of a collection of closed sets. Our main contribution is showing that the computability of the resolvents of the new defined operators is kept with no further assumptions. This result allows us to implement known splitting algorithms under this reformulation, what traduces in the elimination of one variable defining the iterative scheme in comparison to Pierra's approach. 

After the publication of the first preprint version of this manuscript we were noticed about \cite{condat}, where the authors suggest an analogous dimension reduction technique for structured optimization problems. Although that reformulation is different, the derived parallel Douglas--Rachford (DR) algorithm seems to lead to a scheme equivalent to the one obtained from \Cref{th:DR} in this context. Notwithstanding, our analysis is developed in the more general framework of monotone inclusions. Furthermore, we provide detailed proofs of the equivalency and resolvents formulas, as well as numerical comparison to the classical Pierra's reformulation. On the other hand, Malitsky and Tam independently proposed in \cite{MT21} another $r$-operator DR-type algorithm embedded in a reduced-dimensional space. This algorithm, which can be seen as an attempt to extend Ryu's splitting algorithm~\cite{ryu} (see \Cref{rem:MT}), differs from the one proposed in this work and it will also be tested in our experiments.

It is worth mentioning that a similar idea for feasibility problems was previously developed~in~\cite{DDHT21}. In there, the dimensionality reduction was obtained by replacing a pair of constraint sets in the original problem by their intersection before applying Pierra's reformulation. However, the convergence of some  projection algorithms may require a particular intersection structure of these sets. Our approach has the advantage of being directly applicable to any splitting algorithm with no additional requirements.
\pagebreak

The remainder of the paper is organized as follows. In \Cref{sec:prelim} we recall some preliminary notions and auxiliary results. Then \Cref{sec:PS} is divided into \Cref{sec:PS_Pierra}, where we first recall Pierra's standard product space reformulation, and \Cref{sec:PS_new}, in which we propose an alternative reformulation with reduced dimension. We discuss and illustrate the particular case of feasibility and best approximation problems in \Cref{sec:Feas}. In \Cref{sec:Algorithms}, we apply our reformulation to construct new parallel variants of some splitting algorithms. Finally, in \Cref{sec:Experiments} we perform some numerical experiments that exhibit the advantage of the proposed reformulation.

\section{Preliminaries}\label{sec:prelim}

Throughout this paper, $\Hi$ is a Hilbert space endowed with inner product $\langle \cdot,\cdot \rangle$ and induced norm $\|\cdot\|$. We abbreviate \emph{norm convergence} of sequences in $\Hi$ with $\to$ and we use $\wto$ for \emph{weak convergence}.

\subsection{Operators}

Given a nonempty set $D\subseteq\Hi$, we denote by $A:D\tto\Hi$ a \emph{set-valued operator} that maps any point $x\in D$ to a set $A(x)\subseteq\Hi$. In the case where $A$ is single-valued we write  $A:D\to\Hi$. The \emph{graph}, the \emph{domain}, the \emph{range} and the set of \emph{zeros} of A, are denoted, respectively, by $\gra A$, $\dom A$, $\ran A$ and $\zer A$; i.e.,
\begin{gather*}
\gra A:=\left\{(x,u)\in\R^n\times\R^n : u\in A(x)\right\},\quad \dom A:=\left\{x\in\R^n : A(x)\neq\emptyset\right\},\\
\ran A:=\left\{x\in\R^n : x\in A(z) \text{ for some } z\in\R^n  \right\} \quad
\text{and} \quad \zer A:=\left\{x\in\R^n : 0\in A(x)\right\}.
\end{gather*}
The \emph{inverse} of $A$, denoted by $A^{-1}$, is the operator defined via its graph by $\gra A^{-1}:=\{(u,x)\in \R^n \times \R^n: u\in A(x)\}$. We denote the \emph{identity} mapping by $\Id$.

\begin{definition}[Monotonicity]\label{def:mono}
An operator $A:\Hi\tto\Hi$ is said to be
\begin{enumerate}[label=(\roman*)]
\item \emph{monotone} if
\begin{equation*}
\langle x-y,u-v\rangle\geq 0,\quad\forall (x,u),(y,v)\in\gra A;
\end{equation*}
Furthermore, $A$ is said to be \emph{maximally monotone} if it is monotone and there exists no monotone operator $B:\Hi\tto\Hi$ such that $\gra B$ properly contains $\gra A$.
\item \emph{uniformly monotone} with modulus $\phi:\R_{+}\to[0,+\infty]$ if $\phi$ is increasing, vanishes only at 0, and
\begin{equation*}
\langle x-y,u-v\rangle\geq \phi\left(\|x-y\|\right),\quad \forall (x,u),(y,v)\in\gra A.
\end{equation*}
\item $\mu$-\emph{strongly monotone} for $\mu>0$, if $A-\mu\Id$ is monotone; i.e.,
\begin{equation*}
 \langle x-y,u-v\rangle\geq \mu\|x-y\|^2,\quad \forall (x,u),(y,v)\in\gra A.
\end{equation*}
\end{enumerate}
\end{definition}
Clearly, strong monotonicity implies uniform monotonicity, which itself implies monotonicity. The reverse implications are not true.
\begin{remark}
The notions in \Cref{def:mono} can be localized to a subset of the the domain. For instance, $A:\Hi\tto\Hi$ is $\mu$-strongly monotone on $C\subseteq \dom A$ if 
\begin{equation*}
\langle x-y,u-v\rangle\geq \mu\|x-y\|^2,\quad \forall x,y\in C, \forall u\in A(x), \forall v\in A(y).
\end{equation*}
\end{remark}

\pagebreak
\begin{lemma}\label{l:sum_unifstrong}
Let $A,B:\Hi\tto\Hi$ be monotone operators. The following hold.
\begin{enumerate}
\item If $A$ is uniformly monotone on $\dom(A+B)$, then $A+B$ is uniformly monotone with the same modulus than $A$.\label{l:sum_unif}
\item If $A$ is $\mu$-strongly monotone on $\dom(A+B)$, then $A+B$ is $\mu$-strongly monotone.\label{l:sum_strong}
\end{enumerate}
\end{lemma}
\begin{proof}
Let $(x,u),(y,v)\in\gra(A+B)$, i.e., $u=u_1+u_2$ and $v=v_1+v_1$ with $(x,u_1),(y,v_1)\in \gra A$ and $(x,u_2),(y,v_2)\in \gra B$. \ref{l:sum_unif}: Suppose that $A$ is uniformly monotone on $\dom(A+B)$ with modulus $\phi$. Since $x,y\in\dom(A+B)$, we get that
\begin{equation*}
\langle x-y, u-v\rangle = \langle x-y, u_1-v_1\rangle +\langle x-y, u_2-v_2\rangle \geq \phi(\|x-y\|), 
\end{equation*}
which proves that $A+B$ is uniformly monotone with the same modulus. The proof of \ref{l:sum_strong} is analogous and, thus, omitted.
\end{proof}

\begin{definition}[Resolvent]
The \emph{resolvent} of an operator $A:\Hi\tto\Hi$ with parameter $\gamma>0$ is the operator $J_{\gamma A}:\Hi\tto\Hi$ defined by
\begin{equation*}
J_{\gamma A}:=(\Id+\gamma A)^{-1}.
\end{equation*}
\end{definition}

The resolvent of the sum of two monotone operators has no closed expression in terms of the individual resolvents except for some particular situations. The following fact, which is fundamental in our results, contains one of those special cases.

\begin{fact}\label{fact:ressumcomp}
Let $A,B:\Hi\tto\Hi$ be maximally monotone operators such that $B(y)\subseteq B(J_{A}(y))$, for all $y\in\dom B.$
Then, $A+B$ is maximally monotone and
\begin{equation*}\label{eq:ressumcomp}
J_{A+B}(x)=J_A(J_B(x)),\quad \forall x\in\Hi.
\end{equation*}
\end{fact}
\begin{proof}
See,~e.g., \cite[Proposition~23.32(i)]{BC17}.
\end{proof}

\subsection{Functions}

Let $f:\Hi\to{]-\infty,+\infty]}$ be a proper, lower semicontiuous and convex function. The \emph{subdifferential} of $f$  is the operator $\partial f:\Hi\tto\Hi$ defined by
\begin{equation*}
\partial f(x):=\left\{ u\in\Hi :  \langle y-x, u\rangle +f(x)\leq f(y), \quad \forall y\in\Hi\right\}.
\end{equation*}
The \emph{proximity operator} of $f$ (with~parameter~$\gamma$), $\prox_{\gamma f}:\Hi\tto\Hi$, is  defined at $x\in\Hi$ by
\begin{equation*}
\prox_{\gamma f}(x):=\argmin_{u\in\Hi} \left( f(u)+\frac{1}{2\gamma}\|x-u\|^2\right).
\end{equation*}

\begin{fact}\label{f:res_prox}
Let $f:\Hi\to{]-\infty,+\infty]}$ be proper, lower semicontiuous and convex. Then, the subdifferential of $f$, $\partial f$, is a maximally monotone operator whose resolvent becomes the proximity operator of $f$, i.e.,
\begin{equation*}
J_{\gamma \partial f}(x)=\prox_{\gamma f}(x), \quad \forall x\in\Hi.
\end{equation*}
\end{fact}
\begin{proof}
See, e.g., \cite[Theorem~20.25 and Example 23.3]{BC17}.
\end{proof}

\subsection{Sets}
Given a nonempty set $C\subseteq\Hi$, we denote by $d_C$ the \emph{distance function} to $C$; that is, $d_C(x):=\inf_{c\in C}\|c-x\|$, for all $x\in\Hi$. The \index{projection mapping} \emph{projection mapping} (or \emph{projector}) onto $C$ is the possibly set-valued operator $P_C:\Hi\tto C$ defined at each $x\in\Hi$ by
\begin{equation*}
P_C(x):=\left\{p\in C : \|x-p\|=d_C(x)\right\}.
\end{equation*}
Any point $p\in P_C(x)$ is said to be a \emph{best approximation} to $x$ from $C$ (or a \emph{projection} of $x$ onto $C$). If a best approximation in $C$ exists for every point in $\Hi$, then $C$ is said to be \emph{proximinal}. If every point $x\in\Hi$ has exactly one best approximation from $C$, then $C$ is said to be \emph{Chebyshev}. Every nonempty, closed and convex set is Chebyshev (see, e.g., \cite[Theorem 3.16]{BC17}).

The next results characterizes the projection onto a closed affine subspace.

\begin{fact}\label{f:charProj}
Let $D\subseteq\Hi$ be a closed affine subspace and let $x\in\Hi$. Then
$$p=P_D(x) \quad\iff\quad p\in D \,\, \text{ and } \,\, \langle x-p,d-p\rangle= 0,\,\, \forall d\in D.$$
\end{fact}
\begin{proof}
See, e.g., \cite[Corollary 3.22]{BC17}.
\end{proof}

The \emph{indicator function} of a set $C\subseteq \Hi$, $\iota_C:\Hi\to {]-\infty,+\infty]}$, is defined as
$$\iota_C(x):=\left\{\begin{array}{cl}
0, & \text{if } x\in C,\\
+\infty, & \text{if } x\not\in C.\end{array}\right.$$
If $C$ is closed and convex, $\iota_C$ is convex and its differential turns to the \emph{normal cone} to $C$, which is the operator $N_C:\Hi\tto\Hi$ defined by
\begin{equation*}
N_C(x):=\left\{\begin{array}{ll}\{u\in\Hi : \langle u, c-x \rangle\leq 0, \, \forall c\in C \}, &\text{if }x\in C,\\
\emptyset, & \text{otherwise.}\end{array}\right.
\end{equation*}

\begin{fact}\label{f:normalcone}
Let $C\subseteq \Hi$ be  nonempty, closed and convex. Then, the normal cone to $C$, $N_C$, is a maximally monotone operator whose resolvent becomes the projector onto $C$, i.e.,
\begin{equation*}
J_{\gamma N_C}(x)=P_{C}(x), \quad \forall x\in\Hi.
\end{equation*}
\end{fact}
\begin{proof}
See, e.g., \cite[Examples~20.26 and 23.4]{BC17}.
\end{proof}

We conclude this section with the following result that characterizes the projector onto the intersection of a proximinal set (not necessarily convex) and a closed affine subspace under particular assumptions. It is a refinement of \cite[Theorem~3.1(c)]{DDHT21}, whose proof needs to be barely modified.

\begin{lemma}\label{fact:Pinterscomp}
Let $C\subseteq\Hi$ be nonempty and proximinal and let $D\subseteq\Hi$ be a closed affine subspace. If $P_C(d)\cap D\neq \emptyset$  for all $d\in D$, then 
$$P_{C\cap D}(x)=P_C(P_D(x))\cap D,\quad  \forall x\in \Hi.$$
\end{lemma}
\begin{proof}
Fix $x\in\Hi$. By assumption we have that $P_C(P_D(x))\cap D\neq \emptyset$. Pick any $c\in P_C(P_D(x))\cap D$ and let $p\in P_{C\cap D}(x)$. Then $c\in P_C(d)\cap D$, where $d=P_D(x)$. Since $D$ is an affine subspace and $d=P_D(x)$, we derive from \Cref{f:charProj} applied to $c\in D$ and $p\in D$, respectively, that $\langle x-d, c-d\rangle =0$ and $\langle x-d, p-d\rangle =0$. Therefore,
\begin{equation}\label{e:Pinterscomp1}
\|x-c\|^2 = \|x-d\|^2 + \|c-d\|^2 \quad\text{and}\quad \|x-p\|^2 = \|x-d\|^2 + \|p-d\|^2.
\end{equation}
Since $c\in P_C(d)$ and $p\in C$ then $\|c-d\|\leq\|p-d\|$. This combined with \eqref{e:Pinterscomp1} yields $\|x-c\|\leq \|x-p\|$. Note that $p\in P_{C\cap D}(x)$ and $c\in C\cap D$, so it must be
\begin{equation}\label{e:Pinterscomp2}
\|x-c\|=\|x-p\|.
\end{equation}
It directly follows from \eqref{e:Pinterscomp2} that $c\in P_{C\cap D}(x)$. Furthermore, by combining \eqref{e:Pinterscomp2} with \eqref{e:Pinterscomp1} we arrive at $\|c-d\|=\|p-d\|$, which implies that $p\in P_C(d)\cap D$ and concludes the proof.
\end{proof}

\section{Product space reformulation for monotone inclusions}\label{sec:PS}

In this section we introduce our proposed reformulation to convert problems \eqref{prob:zerosum}--\eqref{prob:resolvent}  into equivalent problems with only two operators. To this aim, we first recall the standard product space reformulation due to Pierra~\cite{Pierrathesis,Pierra}.

\subsection{Standard product space reformulation}\label{sec:PS_Pierra}
Consider the product Hilbert space $\Hi^r=\Hi\times\stackrel{(r)}{\cdots}\times\Hi$, endowed with the inner product
\begin{equation*}
\langle \bs x, \bs y \rangle:=\sum_{i=1}^r\langle x_i,y_i\rangle, \quad \forall \bs{x}=(x_1,x_2,\ldots,x_r),\bs{y}=(y_1,y_2,\ldots,y_r)\in\Hi^r,
\end{equation*}
and define
$$\bs D_r:= \left\{ (x,x,\ldots,x)\in \Hi^r : x\in\Hi \right\},$$
which is a closed subspace of $\Hi^r$ commonly known as the \emph{diagonal}. We denote by $\bs{j}_r:\Hi\to\bs{D}_r$ the canonical embedding that maps any $x\in\Hi$ to $\bs{j}_r(x)=(x,x,\ldots,x)\in\bs{D}_r$. 
The following result collects the fundamentals of Pierra's standard product space reformulation.

\begin{fact}[Standard product space reformulation]\label{f:ps}
Let $A_1,A_2,\ldots,A_r:\Hi\tto\Hi$ be maximally monotone and let $\gamma>0$. Define the operator $\bs{A}:\Hi^r\tto\Hi^r$ as
\begin{equation}\label{eq:prod_A}
\bs{A}(\bs{x}):=A_1(x_1)\times A_2(x_2)\times \cdots \times A_r(x_r), \quad \forall \bs{x}=(x_1,x_2,\ldots,x_r)\in\Hi^r.
\end{equation}
Then the following hold.
\begin{enumerate}
\item $\bs A$ is maximally monotone and \label{f:ps_A}
\begin{equation*}
J_{\gamma\bs{A}}(\bs{x})=\left(J_{\gamma A_1}(x_1), J_{\gamma A_2}(x_2),\cdots, J_{\gamma A_r}(x_r)\right), \quad \forall \bs{x}=(x_1,x_2,\ldots,x_r)\in\Hi^r.
\end{equation*}
\item The normal cone to $\bs{D}_r$ is given by \label{f:ps_D}
\begin{equation*}
N_{\bs{D}_r}(\bs{x}) =\left\{\begin{array}{ll}\bs D_{r}^\perp=\{\bs{u}=(u_1,u_2,\ldots,u_r)\in\Hi^r : \sum_{i=1}^r u_i=0 \}, &\text{if }\bs{x}\in \bs{D}_r,\\
\emptyset, & \text{otherwise.}\end{array}\right.
\end{equation*}
It is a maximally monotone operator and
\begin{equation*}
J_{\gamma N_{\bs{D}_r}}(\bs{x})=P_{\bs{D}_r}(\bs{x})=\bs{j}_r\left(\frac{1}{r}\sum_{i=1}^r x_i\right), \quad \forall \bs{x}=(x_1,x_2,\ldots,x_r)\in\Hi^r.
\end{equation*}
\item $\zer\left( \bs{A}+N_{\bs{D}_r}\right)=\bs{j}_r\left(\zer\left(\sum_{i=1}^r A_i\right)\right)$.\label{f:ps_zeros}
\item $J_{\gamma (\bs{A}+N_{\bs{D_r}})}(\bs{x})=\bs{j}_r\left(J_{\frac{\gamma}{r}\sum_{i=1}^{r}A_i}\left( x \right)\right), \quad \forall \bs{x}=\bs j_r(x)\in\bs D_r.$ \label{f:ps_res}
\end{enumerate}
\end{fact}
\begin{proof}
See, e.g.,~\cite[Proposition~26.4]{BC17} and \cite[Proposition~4.1]{aragon2019computing}.
\end{proof}

According to the previous result, the product space reformulation is a convenient trick for reducing problems \eqref{prob:zerosum}--\eqref{prob:resolvent} to equivalent problems with two operators that keep maximal monotonicity and computational tractability. However, this approach relies on working in a product Hilbert space in which each operator of the problem requires one product dimension. 
This may become computationally inefficient when the number of operators increases.
In the next section we will analyze an alternative reformulation in a product Hilbert space with lower dimension. Before that, we include the following technical result regarding additional monotonicity properties that are inherited by the product operator defined in the standard reformulation.   

\begin{lemma}\label{lem:prod_A_unifstrong}
Let $A_1,A_2,\ldots,A_r:\Hi\tto\Hi$ be monotone operators and let $\bs{A}:\Hi^r\tto\Hi^r$ be the product operator defined in \eqref{eq:prod_A}.
Then the following hold.
\begin{enumerate}
\item If $A_{i}$ is uniformly monotone with modulus $\phi_i$ for all $i\in I_0\subseteq\{1,\ldots,r\}$, then $\bs A$ is uniformly monotone on $\dom(\bs A)\cap \bs D_r$ with modulus $\sum_{i\in I_0}\phi_i(\tfrac{\cdot}{\sqrt{r}})$. \label{lem:prod_A_unif}
\item If $A_i$ is $\mu_i$-strongly monotone for all $i\in I_0\subseteq\{1,\ldots,r\}$, then $\bs A$ is $\mu$-strongly monotone on $\dom(\bs A)\cap \bs D_r$ with $\mu:=\frac{1}{r}\sum_{i\in I_0} \mu_i$. \label{lem:prod_A_strong}
\end{enumerate}
\end{lemma}
\begin{proof}
Let $\bs x=\bs j_r(x),\bs y=\bs j_r(y)\in\bs D_r$, for some $x,y\in\Hi$, and let $\bs u=(u1,\ldots,u_r), \bs v=(v_1,\ldots,v_r)\in\Hi^r$ such that $(\bs{x},\bs {u}),(\bs y,\bs v)\in \gra\bs{A}$.

\ref{lem:prod_A_unif}: Suppose that $A_{i}$ is uniformly monotone with modulus~$\phi_i$ for all $i\in I_0\subseteq\{1,\ldots,r\}$.  Then,
\begin{align*}
\langle \bs x-\bs y, \bs u-\bs v\rangle  = \sum_{i=1}^r \langle  x-y, u_i-v_i\rangle  & \geq  \sum_{i\in I_0}\langle  x-y, u_{i_0}-v_{i_0}\rangle \\
&  \geq \sum_{i\in I_0} \phi_i(\|x-y\|) = \sum_{i\in I_0}\phi_i\left(\frac{1}{\sqrt{r}}\|\bs x-\bs y\|\right),
\end{align*}
which implies that $\bs A$ is uniformly monotone on $\dom(\bs A)\cap \bs D_r$ with modulus $\sum_{i\in I_0}\phi_i\left(\frac{\cdot}{\sqrt{r}}\right)$.

\ref{lem:prod_A_strong}: Follows from \ref{lem:prod_A_unif} by taking $\phi_i=\mu_i(\cdot)^2$ for all $i\in I_0$.
\end{proof}

\subsection{New product space reformulation with reduced dimension}\label{sec:PS_new}

We introduce now our proposed reformulation technique which permits to eliminate one space in the product with respect to Pierra's classical trick. More specifically, our approach reformulates problems \eqref{prob:zerosum} and \eqref{prob:resolvent} in the product Hilbert space $$\Hi^{r-1}=\Hi\times\stackrel{(r-1)}{\cdots}\times\Hi.$$ To this aim, consider its diagonal $\bs D_{r-1}$, with canonical embedding $\bs j_{r-1}:\Hi\to\bs D_{r-1}$.

\begin{theorem}[Product space reformulation with reduced dimension]\label{t:nps} Let ${\gamma>0}$ and let $A_1,A_2,\ldots,A_r:\Hi\tto\Hi$ be maximally monotone. Consider the operators $\bs B,\bs K:\Hi^{r-1}\tto\Hi^{r-1}$ defined, at each $\bs{x}=(x_1,\ldots,x_{r-1})\in\Hi^{r-1}$, by
\begin{subequations}\label{eq:nps_sets}
\begin{align}
\bs B(\bs x):= & A_1(x_1) \times \cdots \times A_{r-1}(x_{r-1}),\label{eq:nps_A}\\[1em]
\bs K(\bs x):= & \tfrac{1}{r-1}A_r(x_1) \times {\cdots} \times \tfrac{1}{r-1}A_{r}(x_{r-1}) + N_{\bs{D}_{r-1}}(\bs x).\label{eq:nps_K}
\end{align}
\end{subequations}
Then the following hold.
\begin{enumerate}[label=(\roman*)]
\item $\bs B$ is maximally monotone and \label{t:nps_B}
\begin{equation*}
J_{\gamma \bs{B}}(\bs{x})=\left(J_{\gamma A_1}(x_1), \ldots, J_{\gamma A_{r-1}}(x_{r-1})\right), \quad \forall \bs{x}=(x_1,\ldots,x_{r-1})\in\Hi^{r-1}.
\end{equation*}
\item $\bs K$ is maximally monotone and \label{t:nps_K}
\begin{equation*}
J_{\gamma \bs K}(\bs{x})=\bs{j}_{r-1}\left(J_{\frac{\gamma}{r-1}A_r}\left(\frac{1}{r-1}\sum_{i=1}^{r-1} x_i\right)\right), \quad \forall \bs{x}=(x_1,\ldots,x_{r-1})\in\Hi^{r-1}.
\end{equation*}
If, in addition, $A_r$ is uniformly monotone (resp. $\mu$-strongly monotone), then $\bs K$ is uniformly monotone (resp. $\mu$-strongly monotone).
\item $\zer\left( \bs B+\bs K\right)=\bs{j}_{r-1}\left(\zer\left(\sum_{i=1}^r A_i\right)\right)$.\label{t:nps_zeros}
\item $J_{\gamma (\bs B+\bs K)}(\bs{x})=\bs{j}_{r-1}\left(J_{\frac{\gamma}{r-1}\sum_{i=1}^{r}A_i}\left( x \right)\right), \quad \forall \bs{x}=\bs j_{r-1}(x)\in\bs D_{r-1}.$ \label{t:nps_res}
\end{enumerate}
\end{theorem}
\begin{proof}
Note that \ref{t:nps_B} directly follows from \Cref{f:ps}\ref{f:ps_A}. For the remaining assertions, let us define the operator $\bs S:\Hi^{r-1}\tto\Hi^{r-1}$ as 
\begin{equation*}
\bs S(\bs x):=  \tfrac{1}{r-1}A_r(x_1) \times {\cdots} \times \tfrac{1}{r-1}A_{r}(x_{r-1}), \quad\forall\bs{x}=(x_1,\ldots,x_{r-1})\in\Hi^{r-1},
\end{equation*}
so that $\bs K=\bs S+N_{\bs D_{r-1}}$.

\ref{t:nps_K}: Fix $\bs{x}=(x_1,\ldots,x_{r-1})\in\Hi^{r-1}$. On the one hand, from \Cref{f:ps}\ref{f:ps_A} we get that $\bs S$ is maximally monotone with
\begin{equation*}
J_{\gamma\bs S}(\bs x)= \left(J_{\frac{\gamma}{r-1}A_r}(x_1), \ldots, J_{\frac{\gamma}{r-1}A_r}(x_{r-1})\right).
\end{equation*}
On the other hand, \Cref{f:ps}\ref{f:ps_D} asserts that
\begin{equation*}
N_{\bs{D}_{r-1}}(\bs{x}) =\left\{\begin{array}{ll}\bs D_{r-1}^\perp=\{\bs{u}=(u_1,\ldots,u_{r-1})\in\Hi^{r-1} : \sum_{i=1}^{r-1} u_i=0 \}, &\text{if }\bs{x}\in \bs{D}_{r-1},\\
\emptyset, & \text{otherwise,}\end{array}\right.
\end{equation*}
is maximally monotone with
\begin{equation}\label{fii_e1}
J_{\gamma N_{\bs{D}_{r-1}}}(\bs{x})=P_{\bs{D}_{r-1}}(\bs{x})=\bs{j}_{r-1}\left(\frac{1}{r-1}\sum_{i=1}^{r-1} x_i\right).
\end{equation}
Now pick any $\bs y\in\dom N_{\bs D_{r-1}}=\bs D_{r-1}$. It must be that $\bs y=\bs j_{r-1}(y)$ for some $y\in\Hi$ and thus
\begin{equation}\label{fii_e2}
J_{\gamma \bs S}(\bs y)=\bs j_{r-1}\left(J_{\frac{\gamma}{r-1}A_r}\left(y\right)\right)\in \bs D_{r-1}.
\end{equation}
Hence, we have that
$$ N_{\bs{D}_{r-1}}\left(\bs y \right) = \bs{D}_{r-1}^\perp = N_{\bs{D}_{r-1}}\left(J_{\gamma \bs S}(\bs y)\right).$$
Since $\bs y$ was arbitrary in $\dom N_{\bs D_{r-1}}$ we can apply \Cref{fact:ressumcomp} to obtain that $\bs S + N_{\bs D_{r-1}}$ is maximally monotone and
$$J_{\gamma \bs K}(\bs x)=J_{\gamma (\bs S + N_{\bs{D}_{r-1}})}(\bs{x})=J_{\gamma\bs S}\left(J_{\gamma N_{\bs{D}_{r-1}}}(\bs{x})\right)=
\bs{j}_{r-1}\left(J_{\frac{\gamma}{r-1}A_r}\left(\frac{1}{r-1}\sum_{i=1}^{r-1} x_i\right)\right),$$
where the last equality follows from combining \eqref{fii_e1} and \eqref{fii_e2}.

If, in addition, $A_r$ is uniformly monotone (resp. $\mu$-strongly monotone), then $\bs S$ is uniformly monotone (resp. $\mu$-strongly monotone) on $\dom(\bs S)\cap \bs D_{r-1}$ according to \Cref{lem:prod_A_unifstrong}\ref{lem:prod_A_unif} (resp. \Cref{lem:prod_A_unifstrong}\ref{lem:prod_A_strong}). Since $N_{\bs{D}_{r-1}}$ is a maximally monotone operator with domain $\bs{D}_{r-1}$, the result follows from \Cref{l:sum_unifstrong}\ref{l:sum_unif} (resp. \Cref{l:sum_unifstrong}\ref{l:sum_strong}).

\ref{t:nps_zeros}: To prove the direct inclusion, take any $\bs x \in \zer\left( \bs{B}+\bs{K}\right)= \zer\left( \bs{B}+\bs{S}+N_{\bs{D}_{r-1}}\right)$. It necessarily holds that $\bs x\in \dom N_{\bs D_{r-1}}=\bs D_{r-1}$, so $\bs x=\bs j_{r-1}(x)$ for some $x\in\Hi$. There exist $\bs u\in \bs B(\bs x)$, $\bs v\in \bs S(\bs x)$ and $\bs w \in N_{\bs D_{r-1}}(\bs x)$ with $\bs u+\bs v + \bs w=0$. By definition of these operators $\bs u=(u_1,\ldots,u_{r-1})$, with $u_i\in A_i(x)$ for $i\in\{1,\ldots,r-1\}$, $\bs v= \bs j_{r-1}(\frac{1}{r-1} v)$, with $v\in A_r(x)$, and $\bs w=(w_1,\ldots,w_{r-1})$, with $\sum_{i=1}^{r-1} w_i=0$. Hence,
$$ u_i+\frac{1}{r-1}v+w_i = 0,\quad \text{for each } i\in\{1,\ldots,r-1\}.$$
Summing up all these equations we arrive at
$$0= \sum_{i=1}^{r-1} \left(u_i+\frac{1}{r-1}v+w_i\right) =  \sum_{i=1}^{r-1} u_i + v \in  \sum_{i=1}^{r} A_i (x),$$
which yields $x\in\zer(\sum_{i=1}^r A_i)$.

For the reverse inclusion, take any $x\in\zer(\sum_{i=1}^r A_i)$ and let $\bs x=\bs j_{r-1}(x)\in\bs D_{r-1}$. Then there exists $u_i\in A_i(x)$, for each $i\in\{1,2,\ldots,r\}$, with $\sum_{i=1}^r u_i = 0$. Define
\begin{align*}
\bs u:= & (u_1,\ldots,u_{r-1})\in \bs B(\bs x),\\
\bs v:= & \bs j_{r-1}\left(\frac{1}{r-1}u_r\right)\in \bs S(\bs x),\\
\bs w:= & -\bs u -\bs v \in \bs D_{r-1}^\perp=N_{\bs D_{r-1}}(\bs x).
\end{align*}
Since $\bs u+\bs v+\bs w = \bs 0$ it follows that $\bs x\in \zer\left( \bs{B}+\bs{S}+N_{\bs{D}_{r-1}}\right)$.

\ref{t:nps_res}: Fix any $x\in\Hi$ and let $\bs x=\bs j_{r-1}(x)\in\bs D_{r-1}$ and $\bs p\in J_{\gamma (\bs{B}+\bs{K}+N_{\bs{D}_{r-1}})}(\bs{x})$. Then $$\bs x\in \bs p + \gamma \bs B(\bs p) + \gamma \bs S(\bs p) + N_{\bs{D}_{r-1}}(\bs p).$$
It must be that $\bs p=\bs j_{r-1}(p)\in \bs D_{r-1}$ for some $p\in\Hi$. Hence, we can rewrite the previous inclusion as
\begin{equation}\label{fiv_e1}
x\in p + \gamma A_i(p) + \frac{\gamma}{r-1}A_r(p) + u_i, \quad \text{for each } i\in\{1,\ldots,r-1\},
\end{equation}
with $\sum_{i=1}^{r-1}u_i=0$. Summing up all the inclusions in \eqref{fiv_e1} and dividing by a factor of $r-1$ we arrive at
\begin{equation*}
x\in p + \frac{\gamma}{r-1}\sum_{i=1}^rA_i(p),
\end{equation*}
which implies that $p \in J_{\frac{\gamma}{r-1}\sum_{i=1}^{r}A_i}\left( x \right)$.

For the reverse inclusion, take any $p\in J_{\frac{\gamma}{r-1}\sum_{i=1}^{r}A_i}\left( x \right)$ so that there exist $a_i\in A_i(x)$, for $i\in\{1,2,\ldots,r\}$, such that
\begin{equation}\label{fiv_e2}
x=p+\frac{\gamma}{r-1}\sum_{i=1}^r a_i \quad\iff\quad 0=\sum_{i=1}^{r-1}\left(x-p-\gamma a_i - \frac{\gamma}{r-1}a_r\right).
\end{equation}
Define the vectors $\bs v:=(a_1,\ldots,a_{r-1})$, $\bs w:=\bs j_{r-1}\left(\tfrac{1}{r-1}a_r\right)$ and $\bs u:=\bs x - \bs p - \gamma\bs v - \gamma \bs w$. Hence, $\bs x=\bs p + \gamma\bs v + \gamma\bs w + \bs u$, with $\bs v\in \bs B(\bs x)$, $ \bs w\in\bs S(\bs x)$ and, in view of \eqref{fiv_e2}, $\bs u\in N_{\bs D_{r-1}}$. This implies that $\bs p\in J_{\gamma (\bs{A}+\bs{S}+N_{\bs{D}_{r-1}})}(\bs{x})$ and concludes the proof.
\end{proof}

\begin{remark}\label{rem:recovers}
Consider problem \eqref{prob:zerosum} with only two operators, i.e.,
\begin{equation}\label{prob:zerosum2op}
\text{Find } x\in\Hi \text{ such that } 0 \in A_1(x)+A_2(x),
\end{equation}
where $A_1,A_2:\Hi\tto\Hi$ are maximally monotone. Although splitting algorithms can directly tackle \eqref{prob:zerosum2op}, the product space reformulations are still applicable. Indeed, the standard reformulation in \Cref{f:ps} produces the problem
\begin{equation}\label{prob:zerosum2op_s}
\text{Find } \bs x\in\Hi^2 \text{ such that } 0 \in \bs A(\bs x)+N_{\bs D_2}(\bs x),
\end{equation}
with $\bs A=A_1\times A_2$. Then \eqref{prob:zerosum2op_s} is equivalent to \eqref{prob:zerosum2op} in the sense that their solution sets can be identified to each other. However, they are embedded in different ambient Hilbert spaces. In contrast, the problem generated by applying~\Cref{t:nps} becomes
\begin{equation}\label{prob:zerosum2op_n}
\text{Find } \bs x\in\Hi \text{ such that } 0 \in \bs B(\bs x)+\bs K(\bs x),
\end{equation}
where $\bs B=A_1$ and $\bs K=A_2+N_{\bs D_1}$. Since $\bs D_1=\Hi$, then $N_{\bs D_1}=\{0\}$ and \eqref{prob:zerosum2op_n} recovers the original problem \eqref{prob:zerosum2op}.
\end{remark}

\section{The case of feasibility and best approximation problems}\label{sec:Feas}

Given a family of sets $C_1,C_2,\ldots,C_r\subseteq \Hi$, the \emph{feasibility problem} aims to find a point in the intersection of the sets, i.e.,
\begin{equation}\label{prob:FP}
\text{Find } x \in \bigcap_{i=1}^r C_i.
\end{equation}
A related problem, known as the \emph{best approximation problem}, consists in finding, not only a point in the intersection, but the closest one to a given point $q\in\Hi$, i.e.,
\begin{equation}\label{prob:BP}
\text{Find } p \in \bigcap_{i=1}^r C_i,\quad\text{such that } \|p-q\|=\inf\{\|x-q\|: x\in\cap_{i=1}^r C_i\}.
\end{equation}
The feasibility problem \eqref{prob:FP} can be seen as a particular instance of the monotone inclusion \eqref{prob:zerosum} when specialized to the normal cones to the sets. Indeed, one can easily check that
\begin{equation*}
x \in \bigcap_{i=1}^r C_i \quad\iff\quad 0\in\sum_{i=1}^r N_{C_i}(x).
\end{equation*}
Similarly, under a constraint qualification, problem \eqref{prob:BP} turns out to be \eqref{prob:resolvent} applied to the normal cones, that is,
\begin{equation*}
p\in P_{\cap_{i=1}^r C_i}(q) \quad\iff\quad p\in J_{\sum_{i=1}^r N_{C_i}}(q).\quad\,
\end{equation*}
According to \Cref{f:normalcone}, if the involved sets  $C_1,C_2,\ldots,C_r$ are closed and convex then $N_{C_i}$ is maximally monotone with $J_{N_{C_i}}=P_{C_i}$, for all $i=1,\ldots,r$. Therefore, \Cref{f:ps,t:nps} can be applied in order to reformulate problems \eqref{prob:FP} and \eqref{prob:BP} as equivalent problems involving only two sets. This is illustrated in the following example.

\begin{example}[Convex feasibility problem]\label{ex:convex}
Consider a feasibility problem consisting of finding a point in the intersection of three closed intervals
\begin{equation}\label{e:ex_prob}
\text{Find } x \in C_1\cap C_2\cap C_3\subseteq\R,
\end{equation}
where $C_1:=[0.5,2]$, $C_2:=[1.5,2]$ and $C_3:=[1,3]$. By applying \Cref{f:ps} to the normal cones $N_{C_1}$, $N_{C_2}$ and $N_{C_3}$, the latter is equivalent to
\begin{equation}\label{e:ex_prob_ps}
\text{Find } (x,x,x) \in (C_1\times C_2\times C_3)\cap \bs D_3\subseteq\R^3,\quad \text{where } \bs D_3=\{(x,x,x): x\in\R\}
\end{equation}
In contrast, if we apply \Cref{t:nps} to the normal cones, it can be easily shown that problem \eqref{e:ex_prob} is also equivalent to
\begin{equation}\label{e:ex_prob_nps}
\text{Find } (x,x) \in (C_1\times C_2)\cap \bs K\subseteq\R^2,\quad \text{where } \bs K=\{(x,x): x\in C_3\}.
\end{equation}
Both reformulations are illustrated in \Cref{fig:feas_convex}. Furthermore, the usefulness of the reformulations is that the projectors onto the new sets can be easily computed. Indeed, the projections onto $C_1\times C_2\times C_3$ or $C_1\times C_2$ are computed componentwise in view of \Cref{f:ps}\ref{f:ps_A}, while the projectors onto $\bs D_3$ and $\bs K$ are derived from \Cref{f:ps}\ref{f:ps_D} and \Cref{t:nps}\ref{t:nps_K}, respectively, as
\begin{subequations}
\begin{align}
P_{\bs D_3}(x_1,x_2,x_3)=\left(\frac{x_1+x_2+x_3}{2},\frac{x_1+x_2+x_3}{2},\frac{x_1+x_2+x_3}{2}\right),\\
P_{\bs K}(x_1,x_2)=P_{C_3\times C_3}\left(P_{\bs D_2}(x_1,x_2)\right)=\left(P_{C_3}\left(\frac{x_1+x_2}{2}\right),P_{C_3}\left(\frac{x_1+x_2}{2}\right)\right).\label{eq:ex_c_pk}
\end{align}
\end{subequations}
Observe that, under a constraint qualification guaranteeing the so-called \emph{strong CHIP} holds (i.e. $N_{C_1}+N_{C_3}+N_{C_3}=N_{C_1\cap C_2\cap C_3}$), the reformulations in \eqref{e:ex_prob_ps} and \eqref{e:ex_prob_nps} can also be applied for best approximation problems in view of \Cref{f:ps}\ref{f:ps_res} and \Cref{t:nps}\ref{t:nps_res}, respectively.
\end{example}

\begin{figure}[!htbp]
\centering
  \begin{subfigure}[c]{0.45\textwidth}
  	\vspace{-0.5cm}
     \includegraphics[height=.9\textwidth]{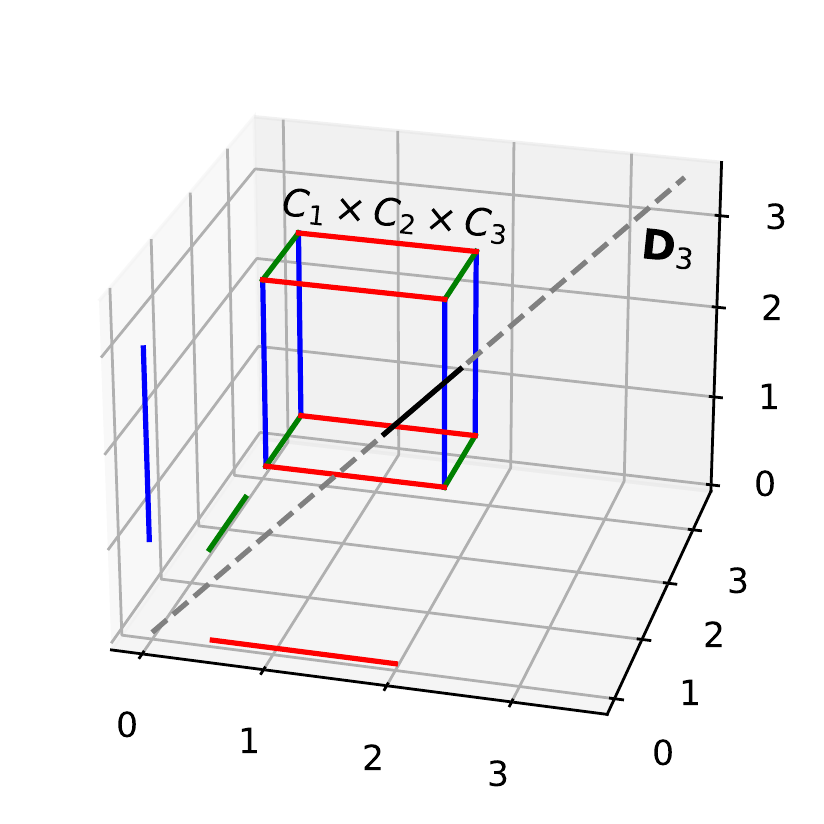}
     \subcaption{Standard product space reformulation}
  \end{subfigure}
  \begin{subfigure}[c]{0.45\textwidth}
     \includegraphics[height=.85\textwidth]{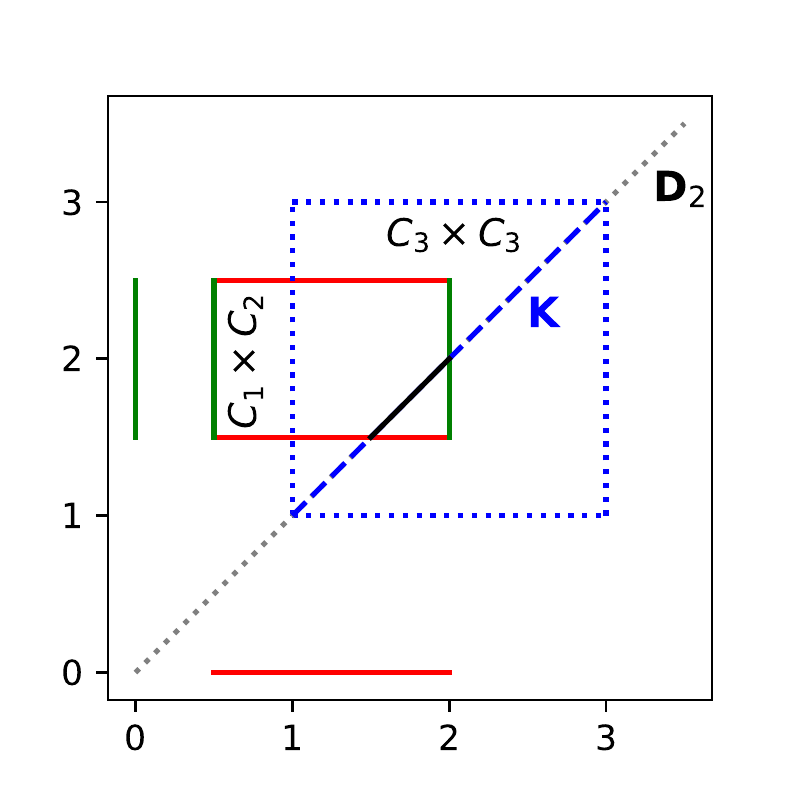}
     \subcaption{Product space reformulation with reduced dimension}
  \end{subfigure}
\caption{Product space reformulations of the convex feasibility problem in \Cref{ex:convex}.}\label{fig:feas_convex}
\end{figure}

Although the theory of projection algorithms is developed under convexity assumptions of the constraint sets, some of them has been shown to be very efficient solvers in a wide variety of nonconvex applications. In special, the Douglas--Rachford algorithm has attracted particular attention due to its well behavior on nonconvex scenarios including some of combinatorial nature; see, e.g.,~\cite{aragon2014recents,aragon2018gc,aragon2020elser,combdes,aragon2020ORclassroom,BCL02,ElserPR, Elser,FHT19,LLS19}. In most of these applications, feasibility problems are described by more than two sets and need to be tackled by Pierra's product space reformulation. Indeed, as we recall in the next result, the reformulation is still valid under the more general assumption that the sets are proximinal but not necessarily convex.

\begin{proposition}[Standard product space reformulation for not necessarily convex feasibility and best approximation problems] Let $C_1,C_2,\ldots,C_r\subseteq \Hi$ be nonempty and proximinal sets and define the product set  \label{p:psf}
\begin{equation}\label{e:psf_set}
\bs C=C_1\times C_2\times\cdots\times C_r \subseteq \Hi^r.
\end{equation}
Then the following hold.
\begin{enumerate}
\item $\bs C$ is proximinal and \label{p:psf_C}
$$P_{\bs C}(\bs x)=P_{C_1}(x_1)\times P_{C_2}(x_2)\times \cdots\times P_{C_r}(x_r),\quad \forall\bs x=(x_1,x_2,\ldots,x_r)\in \Hi^r.$$
If, in addition, $C_1,C_2,\ldots,C_r$ are closed and convex then so is $\bs C$.
\item $\bs D_r$ is a closed subspace with \label{p:psf_D}
$$P_{\bs D_r}(\bs x)=\bs j_{r}\left(\frac{1}{r} \sum_{i=1}^r x_i\right),\quad \forall\bs x=(x_1,x_2,\ldots,x_r)\in \Hi^r.$$
\item $\bs C\cap \bs D_r=\bs j_{r}\left(\cap_{i=1}^r C_i\right)$. \label{p:psf_FP}
\item $P_{\bs C\cap \bs D_r}(\bs x)=\bs j_{r}\left(P_{\cap_{i=1}^r C_i}(x)\right), \quad \forall\bs{x}=\bs{j}_r(x)\in\bs D_{r}$. \label{p:psf_BP}
\end{enumerate}
\end{proposition}
\begin{proof}
\ref{p:psf_C}: Let $\bs x=(x_1,x_2,\ldots,x_r)\in\Hi^r$. By direct computations on the definition of projector we obtain that
\begin{align*}
P_{\bs C}(\bs x)&  = \argmin_{\bs c\in\bs C}  \|\bs x-\bs c\|^2 = \argmin_{(c_1,c_2,\ldots,c_r)\in \bs C} \sum_{i=1}^r\|x_i-c_i\|^2\\
&=\argmin_{c_1\in C_1}\|x_1-c_1\|^2\times \argmin_{c_2\in C_2}\|x_2-c_2\|^2 \times\cdots\times \argmin_{c_r\in C_r}\|x_r-c_r\|^2\\
&=P_{C_1}(x_1)\times P_{C_2}(x_2)\times \cdots\times P_{C_r}(x_r).
\end{align*}
The remaining assertion easily follows from the definition of (topological) product space.

\ref{p:psf_D}: Follows from \Cref{f:ps}\ref{f:ps_D}.

\ref{p:psf_FP}: Let $\bs x\in\bs C\cap \bs D_r$. Then $\bs x=\bs j_r(x)\in\bs D_r$ with $x\in C_i$ for all $i=1,2,\ldots,r$. The reverse inclusion is also straightforward.

\ref{p:psf_BP}: Let $\bs x=\bs j_r(x)\in \bs D_r$. Reasoning as in \ref{p:psf_C} and taking into account \ref{p:psf_FP} we get that
\begin{align*}
P_{\bs C\cap \bs D_r}(\bs x)&= \argmin_{\bs c\in\bs C\cap\bs D_r}\|\bs x-\bs c\|=\bs j_r\left(\argmin_{c\in \cap_{i=1}^rC_i}\|x-c\|\right)=\bs j_r\left(P_{\cap_{i=1}^rC_i}(x)\right).\qedhere
\end{align*}
\end{proof}

Analogously, we show the validity of the product space reformulation with reduced dimension for feasibility and best approximation problems with arbitrary proximinal sets.

\begin{proposition}[Product space reformulation with reduced dimension for non necessarily convex feasibility and best approximation problems] Let $\allowbreak C_1,C_2,\ldots,$ $C_r\subseteq \Hi$ be nonempty and proximinal sets and define\label{p:npsf}
\begin{subequations}\label{eq:npsf_sets}
\begin{align}
\bs B:= & C_1\times\cdots\times C_{r-1} \subseteq \Hi^{r-1},\label{eq:npsf_A}\\[0.2em]
\bs K:= & \{(x,\ldots,x)\in\Hi^{r-1}: x\in C_r\}\subseteq \Hi^{r-1}.\label{eq:npsf_K}
\end{align}
\end{subequations}
Then the following hold.
\begin{enumerate}
\item $\bs B$ is proximinal and \label{p:npsf_B}
$$P_{\bs B}(\bs x)=P_{C_1}(x_1)\times \cdots\times P_{C_{r-1}}(x_{r-1}),\quad \forall\bs x=(x_1,\ldots,x_{r-1})\in \Hi^{r-1}.$$
If, in addition, $C_1,\ldots,C_{r-1}$ are closed and convex then so is $\bs B$.
\item $\bs K$ is proximinal and \label{p:npsf_K}
$$P_{\bs K}(\bs x)=\bs j_{r-1}\left(P_{C_r}\left(\frac{1}{r-1} \sum_{i=1}^r x_i\right)\right),\quad \forall\bs x=(x_1,\ldots,x_{r-1})\in \Hi^{r-1}.$$
If, in addition, $C_{r}$ is closed and convex then so is $\bs K$.
\item $\bs B\cap \bs K=\bs j_{r-1}\left(\cap_{i=1}^{r} C_i\right)$. \label{p:npsf_FP}
\item $P_{\bs B\cap \bs K}(\bs x)=\bs j_{r-1}\left(P_{\cap_{i=1}^r C_i}(x)\right), \quad\forall\bs{x}=\bs{j}_{r-1}(x)\in\bs D_{r-1}$. \label{p:npsf_BP}
\end{enumerate}
\end{proposition}
\begin{proof}
\ref{p:npsf_B}: Follows from \Cref{p:psf}\ref{p:psf_C}.

\ref{p:npsf_K}: First, let us rewrite $$\bs K=\bs j_{r-1}(C_r)=  C_r^{\,r-1} \cap \bs D_{r-1} = (C_r\times\stackrel{(r-1)}{\cdots}\times C_r) \cap \bs{D}_{r-1}\subseteq \Hi^{r-1}.$$
Fix $\bs x=(x_1,\ldots,x_{r-1})\in \Hi^{r-1}.$ By \Cref{p:psf}\ref{p:psf_C} and \ref{p:psf_D}, $C_r^{r-1}$ is a proximinal set and $\bs D_{r-1}$ is a closed subspace with
\begin{align*}
P_{C_r^{r-1}}(\bs x)=P_{C_r}(x_1)\times \cdots\times P_{C_r}(x_{r-1})\quad\text{and}\quad
P_{\bs D_{r-1}}(\bs x)=\bs j_{r-1}\left(\frac{1}{r-1} \sum_{i=1}^{r-1} x_i\right).
\end{align*}
Observe that, for any arbitrary point $\bs y=\bs j_{r-1}(y)\in\bs D_{r-1}$, it holds that
$$\bs j_{r-1}(p)\in P_{C_r^{r-1}}(\bs y)\cap \bs D_{r-1},\quad \forall  p\in P_{C_r}(y).$$
In particular, $P_{C_r^{r-1}}(\bs y)\cap \bs D_{r-1}\neq\emptyset$ for all $\bs y \in \bs{D}_{r-1}$. Hence, by applying \Cref{fact:Pinterscomp} we derive that
\begin{align*}
P_{\bs K}(\bs x)&=P_{C_r^{r-1}\cap \bs D_{r-1}}(\bs x)=P_{C_r^{r-1}}\left(P_{\bs D_{r-1}}(\bs x)\right)\cap \bs D_{r-1}\\
&=P_{C_r^{r-1}}\left(\bs j_{r-1}\left(\frac{1}{r-1} \sum_{i=1}^{r-1} x_i\right)\right)\cap \bs D_{r-1}=\bs j_{r-1}\left(P_{C_r}\left(\frac{1}{r-1} \sum_{i=1}^r x_i\right)\right).
\end{align*}
In addition, if $C_r$ is closed and convex then so is $C_r^{r-1}$ according to \Cref{p:psf}\ref{p:psf_C}. Since $\bs D_{r-1}$ is a closed subspace, the convexity and closedness of $\bs K$ follows.

\ref{p:npsf_FP} and \ref{p:npsf_BP}: Their proofs are straightforward and analogous to the proofs of \Cref{p:psf}\ref{p:psf_FP} and \ref{p:psf_BP}, respectively, so they are omitted.
\end{proof}

\begin{example}[Nonconvex feasibility problem]\label{ex:nonconvex}
Consider the feasibility problem
\begin{equation}\label{e:ex_prob_nc}
\text{Find } x \in C_1\cap C_2\cap \widehat C_3\subseteq\R,
\end{equation}
where $C_1:=[0.5,2]$, $C_2:=[1.5,2]$ and $\widehat C_3:=\{1,2,3\}$; that is, the problem considered in \Cref{ex:convex} but replacing $C_3$ by the nonconvex set $\widehat C_3$. According to \Cref{p:psf,p:npsf}, the product space reformulations in \eqref{e:ex_prob_ps} and \eqref{e:ex_prob_nps}, with $C_3$ replaced by $\widehat{C}_3$, are still valid to reconvert \eqref{e:ex_prob_nc} into an equivalent problem described by two sets. Both formulations are illustrated in \Cref{fig:feas_nonconvex}, where now we denote
$$\widehat{\bs K}=\{(x,x): x\in\widehat{C}_3\}=(\widehat{C}_3\times\widehat{C}_3)\cap\bs D_2=\{(1,1), (2,2), (3,3)\}.$$
Due to the nonconvexity, the projector onto $\widehat C_3$ may be set-valued. In view of \Cref{p:npsf}\ref{p:npsf_K}, the projector onto $\widehat{\bs K}$ is described by
\begin{equation*}
P_{\widehat{\bs K}}(x_1,x_2)=\left\{(p,p): p\in P_{\widehat{C}_3}\left(\frac{x_1+x_2}{2}\right)\right\}.
\end{equation*}
We emphasize that, in contrast to \eqref{eq:ex_c_pk}, in the nonconvex case $P_{\widehat{\bs K}}\neq P_{\widehat{C}_3\times \widehat{C}_3}\circ P_{\bs D_2}$. Indeed, consider for instance the point $\bs x:=(2,1)\in\R^2$. Then,
\begin{align*}
P_{\widehat{\bs K}}(\bs x)&=\{(1,1), (2,2)\},\\
P_{\widehat{C}_3\times \widehat{C}_3}(P_{\bs D_2}(\bs x))&=P_{\widehat{C}_3\times \widehat{C}_3}( (1.5,1.5))=\{(1,1),(1,2),(2,1),(2,2)\}.
\end{align*}
Therefore, $P_{\widehat{C}_3\times \widehat{C}_3}(P_{\bs D_2}(\bs x))\neq P_{\widehat{\bs K}}(\bs x)=P_{\widehat{C}_3\times \widehat{C}_3}(P_{\bs D_2}(\bs x))\cap \bs D_2$.
\end{example}

\begin{figure}[!htbp]
\centering
  \begin{subfigure}[c]{0.45\textwidth}
  \vspace{-0.8cm}
     \includegraphics[height=.95\textwidth]{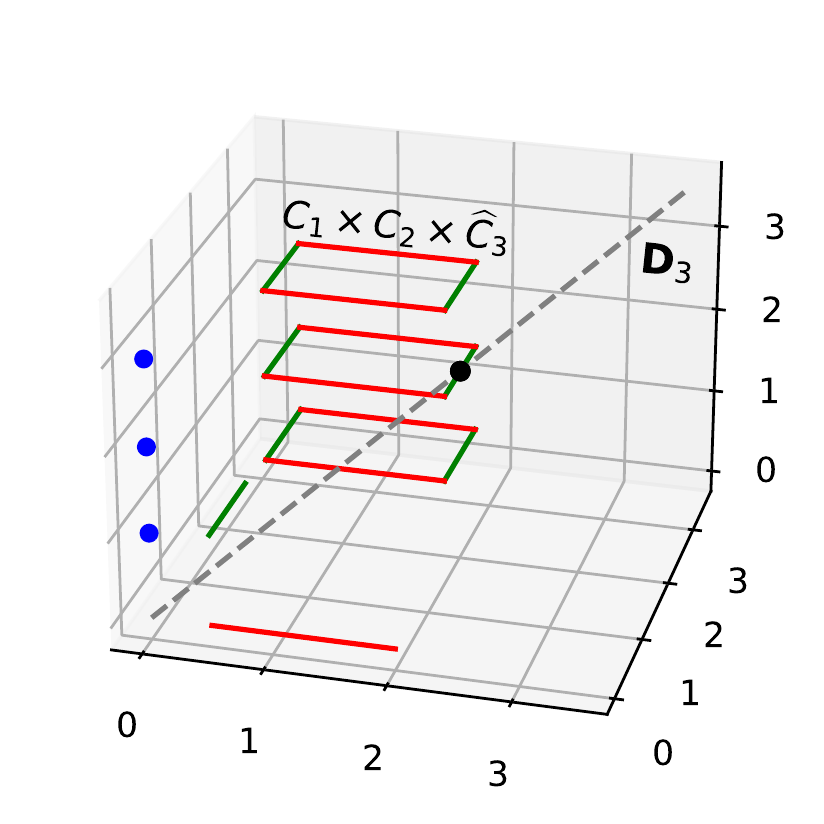}
     \vspace{-0.2cm}
     \subcaption{Standard product space reformulation}
  \end{subfigure}
  \begin{subfigure}[c]{0.45\textwidth}
     \includegraphics[height=.85\textwidth]{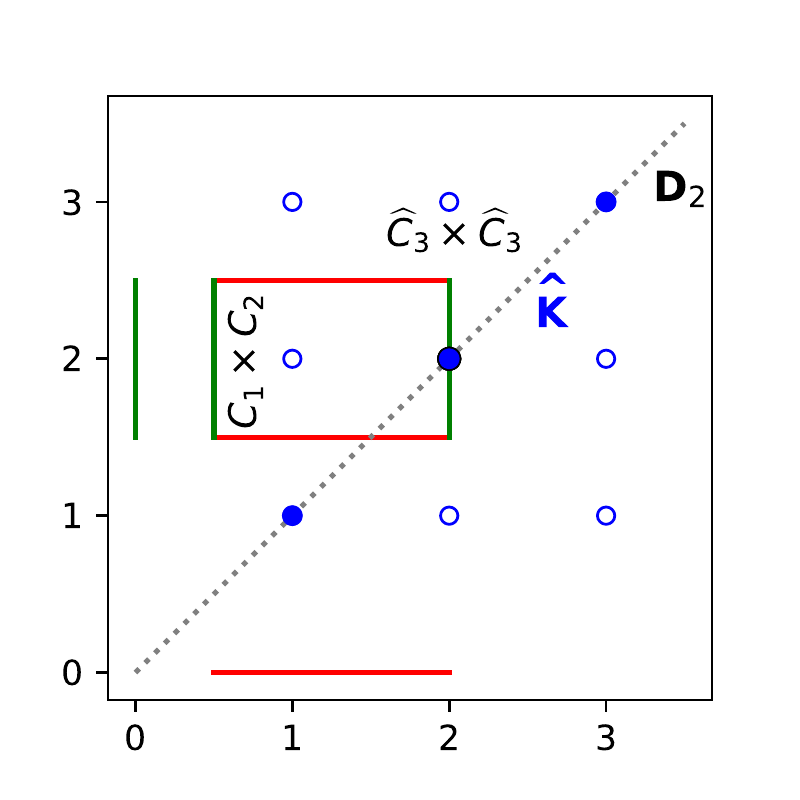}
     \subcaption{Product space reformulation with reduced dimension}
  \end{subfigure}
\caption{Product space reformulations of the nonconvex feasibility problem in \Cref{ex:nonconvex}.}\label{fig:feas_nonconvex}
\end{figure}

\section{Application to splitting algorithms}\label{sec:Algorithms}

In this section, we apply our proposed reformulation in \Cref{t:nps} in order to derive two new parallel splitting algorithms, one for solving problem \eqref{prob:zerosum}, and another one for \eqref{prob:resolvent}. In the first case, we consider the \emph{Douglas--Rachford} (DR) algorithm~\cite{DR56,LM79} (see also \cite{BM17,BM21} for recent results in the inconsistent case). The DR algorithm permits to find a zero of the sum of two maximally monotone operators. When it is applied to Pierra's standard reformulation the resulting method takes the form in \cite[Proposition~26.12]{BC17}. In contrast, if the problem is reformulated via~\Cref{t:nps} we obtain the following iterative scheme, which requires one variable less.

\begin{theorem}[Parallel Douglas/Peaceman--Rachford splitting algorithm]\label{th:DR}
Let $\allowbreak A_1, A_2, \ldots, A_r:\Hi\tto \Hi$ be maximally monotone operators such that $\zer(\sum_{i=1}^r A_i)\neq\emptyset$. Let $\gamma>0$ and let $\lambda\in{]}0,2]$. Given $x_{1,0},\ldots,x_{r-1,0}\in\Hi$, set
\begin{equation}\label{it:DR}
\begin{aligned}
&\text{for } k=0, 1, 2, \ldots:\\
&\left\lfloor\begin{array}{l}
p_k=J_{\frac{\gamma}{r-1}A_r}\left(\frac{1}{r-1}\sum_{i=1}^{r-1} x_{i,k}\right),\\
\text{for } i=1, 2,\ldots, r-1:\\
\left\lfloor\begin{array}{l}
z_{i,k}=J_{\gamma A_i}\left(2p_k-x_{i,k}\right),\\
x_{i,k+1}=x_{i,k}+\lambda\left( z_{i,k}-p_k\right).
\end{array}\right.
\end{array}\right.
\end{aligned}
\end{equation}
Then the following hold.
\begin{enumerate}
\item If $\lambda\in{]0,2[}$, then  $p_k\wto p^\star$ and $z_{i,k}\wto p^\star$, for $i=1,\ldots,r-1$, with $p^\star\in\zer(\sum_{i=1}^r A_i)$.\label{th:DR_I}
\item If $A_r$ is uniformly monotone, then  $p_k\to p^\star$ and $z_{i,k}\to p^\star$, for $i=1,\ldots,r-1$, where $p^\star$ is the unique point in $\zer(\sum_{i=1}^r A_i)$.\label{th:DR_II}
\end{enumerate}
\end{theorem}
\begin{proof}
Consider the product Hilbert space $\Hi^{r-1}$ and let $\bs B,\bs K:\Hi^{r-1}\tto\Hi^{r-1}$ be the operators defined in \eqref{eq:nps_sets}. By \Cref{t:nps}\ref{t:nps_B}, \ref{t:nps_K} and \ref{t:nps_zeros}, we get that $\bs B$ and $\bs K$ are maximally monotone with $\zer(\bs B+\bs K)=\bs j_{r-1}(\zer(\sum_{i=1}^r A_i))\neq\emptyset$. For each $k=0,1,2,\ldots$, set $\bs x_{k}:=(x_{1,k},\ldots,x_{r-1,k}),\bs z_{k}:=(z_{1,k},\ldots,z_{r-1,k})\in\Hi^{r-1}$ and $\bs p_k=\bs j_{r-1}(p_k)\in\bs D_{r-1}$. Hence, according to \Cref{t:nps}\ref{t:nps_B} and \ref{t:nps_K}, we can rewrite \eqref{it:DR} as 
\begin{equation}\label{it:DR_nps}
\begin{aligned}
&\text{for } k=0, 1, 2, \ldots:\\
&\left\lfloor\begin{array}{l}
\bs p_k=J_{\gamma \bs K}(\bs x_k),\\
\bs z_{k}=J_{\gamma \bs B}\left(2\bs p_k-\bs x_{k}\right),\\
\bs x_{k+1}=\bs x_{k}+\lambda\left( \bs z_{k}-\bs p_k\right).
\end{array}\right.
\end{aligned}
\end{equation}
Note that \eqref{it:DR_nps} is the Douglas--Rachford (or Peaceman--Rachford) iteration applied to the operators $\bs B$ and $\bs K$. If $\lambda\in{]0,2[}$, we apply \cite[Theorem~26.11(iii)]{BC17} to obtain that $\bs p_k\wto\bs p^\star$ and $\bs z_k\wto\bs p^\star$, with $\bs p^\star\in\zer(\bs B+\bs K)$. Hence, $\bs p^\star=\bs j_{r-1}(p^\star)$ with $p^\star\in \zer\left(\sum_{i=1}^r A_i\right)$, which implies \ref{th:DR_I}. 

Suppose in addition that $A_r$ is uniformly monotone. Then so is $\bs K$ according to \Cref{t:nps}\ref{t:nps_K}. Hence, \ref{th:DR_II} follows from \cite[Theorem~26.11(vi)]{BC17}, when $\lambda\in{]0,2[}$, and \cite[Proposition~26.13]{BC17} when $\lambda=2$. 
\end{proof}

\begin{remark}[Frugal resolvent splitting algorithms with minimal lifting] \label{rem:MT}
Consider the problem of finding a zero of the sum of three maximally monotone operators $A,B,C:\Hi\tto\Hi$. The classical procedure to solve it has been to employ the standard product space reformulation (\Cref{f:ps}) to construct a DR algorithm on $\Hi^3$. The question of whether it is possible to generalize the DR algorithm to three operators without \emph{lifting}, that is, without enlarging the ambient space, was solved with a negative answer by Ryu in~\cite{ryu}. The generalization is considered in the sense of devising a \emph{frugal} splitting algorithm which uses the resolvent of each operator exactly once per iteration. In the same work, the author demonstrated that the minimal lifting is $2$-fold (in $\Hi^2$) by providing the following splitting algorithm. Given $\lambda\in{]0,1[}$ and $x_0,y_0\in\Hi$, set
\begin{equation}\label{it:Ryu}
\begin{aligned}
&\text{for } k=0, 1, 2, \ldots:\\
&\left\lfloor\begin{array}{l}
u_k = J_{\gamma A}(x_k)\\
v_k = J_{\gamma B}(u_k+y_k)\\
w_k = J_{\gamma C}(u_k-x_k+v_k-y_k)\\
x_{k+1} = x_k + \lambda(w_k-u_k)\\
y_{k+1} = y_k + \lambda(w_k-v_k).
\end{array}\right.
\end{aligned}
\end{equation}
Then $u_k\wto w^\star$, $v_k\wto w^\star$ and $w_k\wto w^\star$, with $w^\star\in\zer(A+B+C)$ (see \cite[Theorem~4]{ryu} or \cite[Appendix A]{aragon2020strengh} for an alternative proof in an infinite-dimensional space).

A few days after the publication of the first preprint version of this manuscript ({ArXiv: \href{https://arxiv.org/abs/2107.12355}{2107.12355}}), Malitsky and Tam \cite{MT21} generalized Ryu's result by showing that for an arbitrary number of $r$ operators the minimal lifting is $(r-1)$-fold. In addition, they proposed another frugal splitting algorithm that attains this minimal lifting, whose iteration is described as follows. Given $\lambda\in{]0,1[}$ and $\bs z_0=(z_{1,0},\ldots,z_{r-1,0})\in\Hi^{r-1}$, set
\begin{equation}\label{it:MT}
\begin{aligned}
&\text{for } k=0, 1, 2, \ldots:\\
&\left\lfloor\begin{array}{l}
\text{Compute } \bs z_{k+1}=(z_{1,k+1},\ldots,z_{r-1,k+1})\in\Hi^{r-1} \text{ as }\\
\bs z_{k+1}=\bs z_k + \lambda \left(\begin{array}{c}
x_{2,k}-x_{1,k}\\
x_{3,k}-x_{2,k}\\
\vdots\\
x_{r,k}-x_{r-1,k}
\end{array}\right),\\
\text{where } \bs x_k=(x_{1,k},x_{2,k},\ldots,x_{r,k})\in\Hi^r \text{ is given by }\\
\hspace{2ex}x_{1,k}=J_{\gamma A_1}(z_{1,k}),\\
\hspace{2ex}\text{for } k=2,\ldots,r-1:\\
\hspace{2ex}\left\lfloor\begin{array}{l}
\hspace{2ex}x_{i,k}=J_{\gamma A_i}(z_{i,k}-z_{i-1,k}+x_{i-1,k}),
\end{array}\right.\\
\hspace{2ex}x_{r,k}=J_{\gamma A_r}(x_{1,k}+x_{r-1,k}-z_{r-1,k}).
\end{array}\right.
\end{aligned}
\end{equation}
Then,  for each $i\in\{1,\ldots,r\}$, $x_{i,k}\wto x^*\in \zer(\sum_{j=1}^r A_j)$ (see \cite[Theorem 4.5]{MT21}).

It is worth to notice that the Malitsky--Tam iteration \eqref{it:MT} does not generalize Ryu's scheme \eqref{it:Ryu}, which seems to be difficult to extend to more than three operators as explained in \cite[Remark~4.7]{MT21}. Furthermore, both of these algorithms are different from the one in \Cref{th:DR}. The main conceptual difference is that \eqref{it:MT} can be implemented in a distributed decentralized way whereas algorithm \eqref{it:DR} uses the operator $A_r$ as a central coordinator (see~\cite[\S~5]{MT21}). Nevertheless, for the applications considered in this work, the dimensionality reduction obtained through the new product space reformulation seems to be more effective for accelerating the converge of the algorithm, especially when the number of operators is large as we shall show in \Cref{sec:Experiments}.
\end{remark}

We now turn our attention into splitting algorithms for problem \eqref{prob:resolvent}. In particular, we concern on the \emph{averaged alternating modified reflections (AAMR) algorithm}, originally proposed in \cite{aragon2018new} for best approximation problems, and later extended in \cite{aragon2019computing} for monotone operators (see also~\cite{alwadani2018asymptotic,aragon2020strengh}). The parallel AAMR splitting iteration obtained from Pierra's reformulation is given in~\cite[Theorem~4.1]{aragon2019computing}. As we show in the following result, we can avoid one of the variables defining the iterative scheme if we use the product space reformulation~in~\Cref{t:nps}.

\begin{theorem}[Parallel AAMR splitting algorithm]\label{th:par_AAMR}
Let $\allowbreak A_1, A_2, \ldots, A_r:\Hi\tto \Hi$ be maximally monotone operators, let $\gamma>0$ and let $\lambda\in{]0,2]}$. Let $\beta\in{]0,1[}$ and suppose that $q\in\ran\left(\Id+\frac{\gamma}{2(1-\beta)(r-1)}\sum_{i=1}^rA_i\right)$. Given $x_{1,0},\ldots,x_{r-1,0}\in\Hi$, set
\begin{equation}\label{it:AAMR}
\begin{aligned}
&\text{for } k=0, 1, 2, \ldots:\\
&\left\lfloor\begin{array}{l}
p_k=J_{\frac{\gamma}{r-1}A_r}\left(\frac{\beta}{r-1}\sum_{i=1}^{r-1} x_{i,k}+(1-\beta)q\right),\\
\text{for } i=1, 2,\ldots, r-1:\\
\left\lfloor\begin{array}{l}
z_{i,k}=J_{\gamma A_i}\left(\beta(2p_k-x_{i,k})+(1-\beta)q\right),\\
x_{i,k+1}=x_{i,k}+\lambda\left( z_{i,k}-p_k\right).
\end{array}\right.
\end{array}\right.
\end{aligned}
\end{equation}
Then $\left(p_k\right)_{k=0}^\infty$ converges strongly to $J_{\frac{\gamma}{2(1-\beta)(r-1)}\sum_{i=1}^rA_i}(q)$.
\end{theorem}
\begin{proof}
Consider the product Hilbert space $\Hi^{r-1}$ and let $\bs B,\bs K:\Hi^{r-1}\tto\Hi^{r-1}$ be the operators defined in \eqref{eq:nps_sets}. We know that $\bs B$ and $\bs K$ are maximally monotone by \Cref{t:nps}\ref{t:nps_B} and \ref{t:nps_K}, respectively. Set $\bs x_{k}:=(x_{1,k},\ldots,x_{r-1,k}),\bs z_{k}:=(z_{1,k},\ldots,z_{r-1,k})\in\Hi^r$ and $\bs p_k=\bs j_{r-1}(p_k)\in\bs D_{r-1}$, for each $k=0,1,2,\ldots$, and set $\bs q:=\bs j_{r-1}(q)\in\bs D_{r-1}$. On the one hand, according to \Cref{t:nps}\ref{t:nps_B} and \ref{t:nps_K}, we can rewrite \eqref{it:AAMR} as
\begin{equation*}
\begin{aligned}
&\text{for } k=0, 1, 2, \ldots:\\
&\left\lfloor\begin{array}{l}
    \bs p_k = J_{\gamma \bs K}\left(\beta \bs x_k + (1-\beta)\bs q\right) \\
    \bs z_k = J_{\gamma \bs B}\left(\beta(2\bs p_k-\bs x_k)+(1-\beta)\bs q\right) \\
    \bs x_{k+1} = \bs x_k + {\lambda}(\bs z_k-\bs p_k). 
\end{array}\right.
\end{aligned}
\end{equation*}
On the other hand, from \Cref{t:nps}\ref{t:nps_res} we obtain that
\begin{equation*}\label{eq:ressum_beta}
J_{\frac{\gamma}{2(1-\beta)} (\bs B+\bs K)}(\bs{q})=\bs{j}_{r-1}\left(J_{\frac{\gamma}{2(1-\beta)(r-1)}\sum_{i=1}^{r}A_i}\left( q \right)\right).
\end{equation*}
In particular, the latter implies that $\bs q\in\ran\left(\Id+ \frac{\gamma}{2(1-\beta)} (\bs B+\bs K)\right)$. Hence, by applying \cite[Theorem~6 and Remark 10(i)]{aragon2020strengh}, we conclude that $\left(\bs p_k\right)_{k=0}^\infty$ converges strongly to $J_{\frac{\gamma}{2(1-\beta)} (\bs B+\bs K)}(\bs{q})$ and the result follows.
\end{proof}

\begin{remark}[On Forward-Backward type methods] Forward-Backward type methods permit to find a zero in $A+B$ when $A:\Hi\to\Hi$ is cocoercive (see, e.g., \cite[Theorem~26.14]{BC17}) or Lipschitz continuous (see, e.g., \cite{cevher2019reflected,malitsky2018forward,tseng2000modified}) and $B:\Hi\tto\Hi$ is maximally monotone. These algorithms make use of direct evaluations of $A$ (forward steps) and resolvent computations of $B$ (backward steps). When dealing with finitely many operators of both nature (single-valued and set-valued), Pierra's reformulation (\Cref{f:ps}) yields parallel algorithms which need to activate all of them through their resolvents, since all of them are combined into the product operator $\bs A$ in \eqref{eq:prod_A}.
In contrast, the product space reformulation in \Cref{t:nps} allows to deal with the case when $A_1,\ldots,A_{r-1}:\Hi\to\Hi$ are cocoercive/Lipschitz continuous and $A_r:\Hi\tto\Hi$ is maximally monotone. Indeed, it can be easily proved that the product operator $\bs B$ in \eqref{eq:nps_A} keeps the cocoercivity/Lipschitz continuity property. However, the parallel algorithm obtained with this approach will coincide with the original Forward-Backward type algorithm applied to the operators $\sum_{i=1}^{r-1}A_i$ and $A_r$. It is worth mentioning that in the opposite case, that is, when one operator is cocoercive and the remaining ones are maximally monotone, a parallel Forward-Backward algorithm was developed in \cite{Raguet}.
\end{remark}

\section{Numerical experiments}\label{sec:Experiments}

In this section, we perform some numerical experiments to assess the advantage of the new proposed reformulation when applied to splitting or projection algorithms. In particular, we compare the performance of the proposed parallel Douglas--Rachford algorithm in \Cref{th:DR} with the standard parallel version in \cite[Proposition~26.12]{BC17}, first on a convex minimization problem and then in a nonconvex feasibility problem. We will refer to these algorithms as \emph{Reduced-DR} and \emph{Standard-DR}, respectively. In some experiments we will also test the algorithms in \cite[Theorem~4]{ryu} and \cite[Theorem 4.5]{MT21}, wich will be referred to as \emph{Ryu} and \emph{Malitsky--Tam}, respectively.  All codes were written in Python 3.7 and the tests were run on an Intel Core i7-10700K CPU \@3.80GHz with 64GB RAM, under Ubuntu 20.04.2 LTS (64-bit).

\subsection{The generalized Heron problem}

We first consider the \emph{generalized Heron problem}, which is described as follows. Given $\Omega_1,\ldots,\Omega_r\subseteq \R^n$ nonempty, closed and convex sets, we are interested in finding a point in $\Omega_r$ that minimizes the sum of the distances to the remaining sets; that is, 
\begin{equation}\label{p:Heron}
\begin{array}{rl}
\text{Min} & \sum_{i=1}^{r-1} d_{\Omega_i}(x)\\
\text{s.t.} & x\in \Omega_r.
\end{array}
\end{equation}%
This problem was investigated with modern convex analysis tools in \cite{MNS12a,MNS12b}, where it was solved by subgradient-type algorithms. It was later revisited in \cite{BH13}, where the authors implemented their proposed paralellized Douglas--Rachford-type primal-dual methods for its resolution. Indeed, splitting algorithms such as Douglas--Rachford can be employed to solve problem \eqref{p:Heron} as this is equivalent to the monotone inclusion \eqref{prob:zerosum} with
\begin{equation*}
A_r=\partial\iota_{\Omega_r}=N_{\Omega_r} \quad\text{and}\quad A_i=\partial d_{\Omega_i}, \text{ for } i=1,\ldots,r-1. 
\end{equation*}
According to \Cref{f:res_prox,f:normalcone}, $J_{\gamma A_r}=P_{\Omega_r}$ and $J_{\gamma A_i}=\prox_{\gamma d_{\Omega_i}}$, for $i=1,\ldots,r-1$. We recall that the proximity operator of the distance function to a closed and convex set $C\subseteq\Hi$ is given by
\begin{equation*}
\prox_{\gamma d_{C}}(x)=\left\{\begin{array}{ll}
x+\frac{\gamma}{{d_{C}(x)}}\left({P_{C}(x)-x}\right), & \text{if } d_{C}(x)>\gamma,\\
P_{C}(x), & \text{otherwise}.\end{array}\right.
\end{equation*}

In our experiments, the constraint sets $\Omega_1,\ldots,\Omega_{r-1}$ in \eqref{p:Heron} were randomly generated hypercubes of centers $(c_{i,1},\ldots,c_{i,n}),\ldots,(c_{r-1,1},\ldots,c_{r-1,n})\in\R^n$ with length side $\sqrt{2}$, while $\Omega_r$ was chosen to be the closed ball centered at zero with radius $10$; that is,
\begin{subequations}\label{eq:Heron_const}
\begin{align}
\Omega_i&:=\left\{(x_1,\ldots,x_n)\in\R^n : |c_{i,j}-x_j|\leq \frac{\sqrt{2}}{2}, \, j=1,\ldots, n\right\}, \quad i=1,\ldots,r-1,\\
\Omega_r&:=\{ x\in \R^n : \|x\|\leq 10\}.
\end{align}
\end{subequations}
More precisely, the centers of the hypercubes were randomly generated with norm greater or equal than $12$, so that the hypercubes did not intersect the ball. Two instances of the problem with $r=5$, in $\R^2$ and $\R^3$, are illustrated in \Cref{fig:Heron}.

\begin{figure}[!htbp]
\centering
\vspace{-0.3cm}
  \hspace{0.08\textwidth}\begin{subfigure}[c]{0.35\textwidth}
     \includegraphics[width=.85\textwidth]{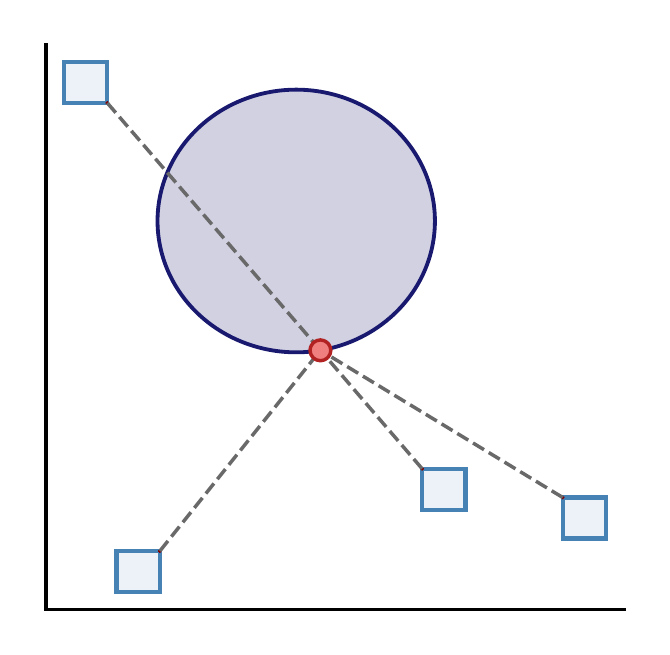}
  \end{subfigure}\hspace{0.06\textwidth}
  \begin{subfigure}[c]{0.45\textwidth}
     \includegraphics[width=.85\textwidth]{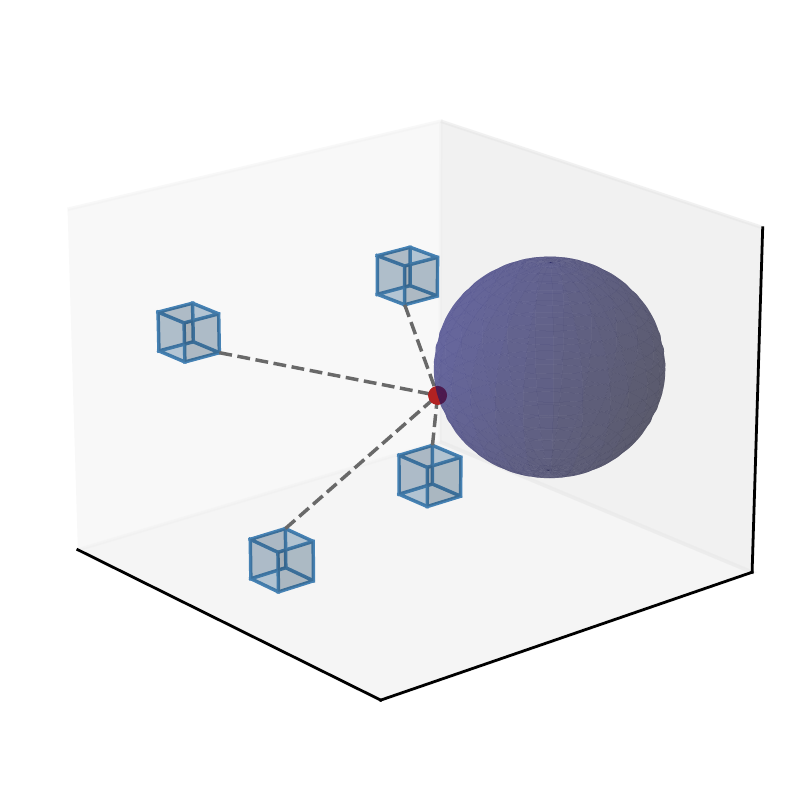}
  \end{subfigure}
  \vspace{-.8cm}
\caption{The generalized Heron problem consisting in finding a point in a ball in $\R^2$ (left) or $\R^3$ (right) that minimizes the sum of the distances to four squares (left) or cubes (right). A solution to the problem is represented by a red point.}\label{fig:Heron}
\end{figure}

In our first numerical test, we generated 10 instances of the problem \eqref{p:Heron}--\eqref{eq:Heron_const} in $\R^{100}$ with $r=3$. For each $\gamma\in\{1,10,25,50,75,100\}$ and each $\lambda\in\{0.1,0.2,\ldots,1.9\}$, Standard-DR and Reduced-DR were run from $10$ random starting points. For those values of $\lambda\leq 1$, Ryu and Malitsky--Tam algorithms were also run from the same initial points. All algorithms were stopped when the monitored sequence $\{p_k\}_{k=0}^\infty$ verified the Cauchy-type stopping criteria
\begin{equation*}
\|p_{k+1}-p_k\|<\varepsilon:= 10^{-6}
\end{equation*}
for the first time. For a fairer comparison, for each algorithm we monitored that sequence which is projected onto the feasible set $\Omega_r$ so that all of them lay on the same ambient space. The average number of iterations required by each algorithm among all problems and starting points is depicted in \Cref{fig:Heron_iter}. In \Cref{tbl:Heron} we list the best results obtained by each algorithm and the value of the parameters at which those results were achieved.

\begin{figure}[!ht]
\centering
  \begin{subfigure}[c]{0.44\textwidth}
     \includegraphics[width=\fw\textwidth]{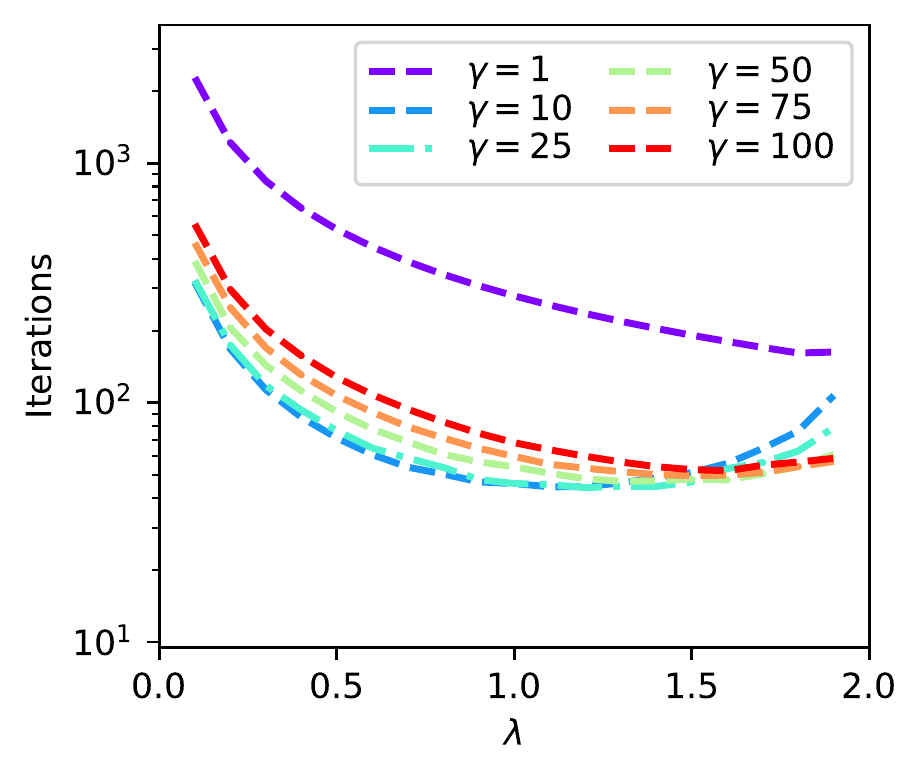}
     \subcaption{Standard-DR}
  \end{subfigure}
  \begin{subfigure}[c]{0.44\textwidth}
     \includegraphics[width=\fw\textwidth]{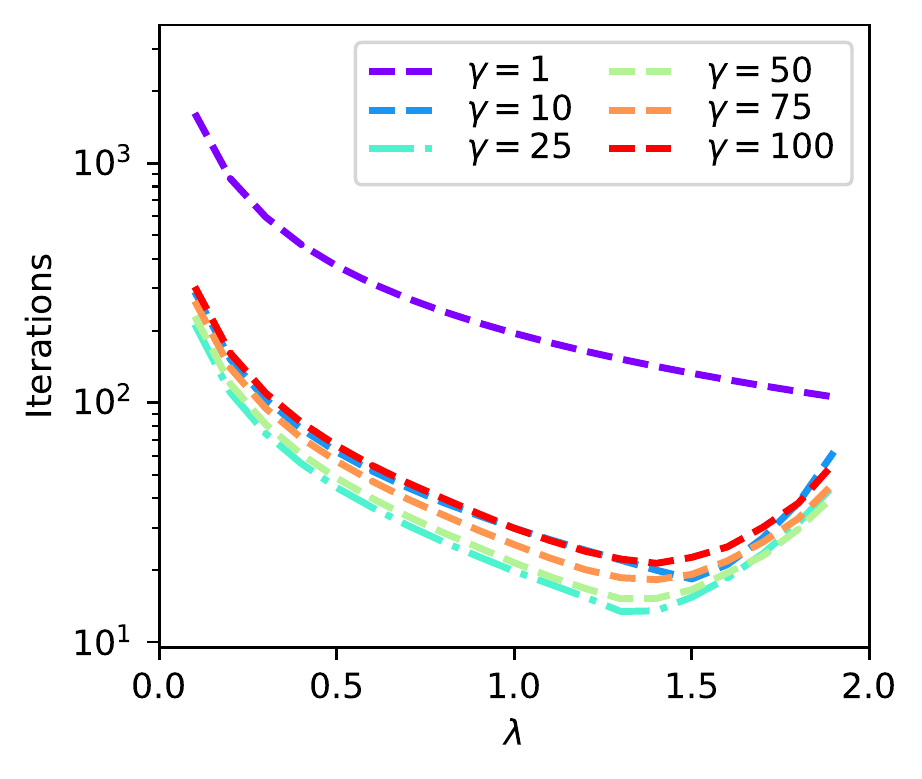}
     \subcaption{Reduced-DR}
  \end{subfigure}\\[3ex]
  
   \begin{subfigure}[c]{0.44\textwidth}
     \includegraphics[width=\fw\textwidth]{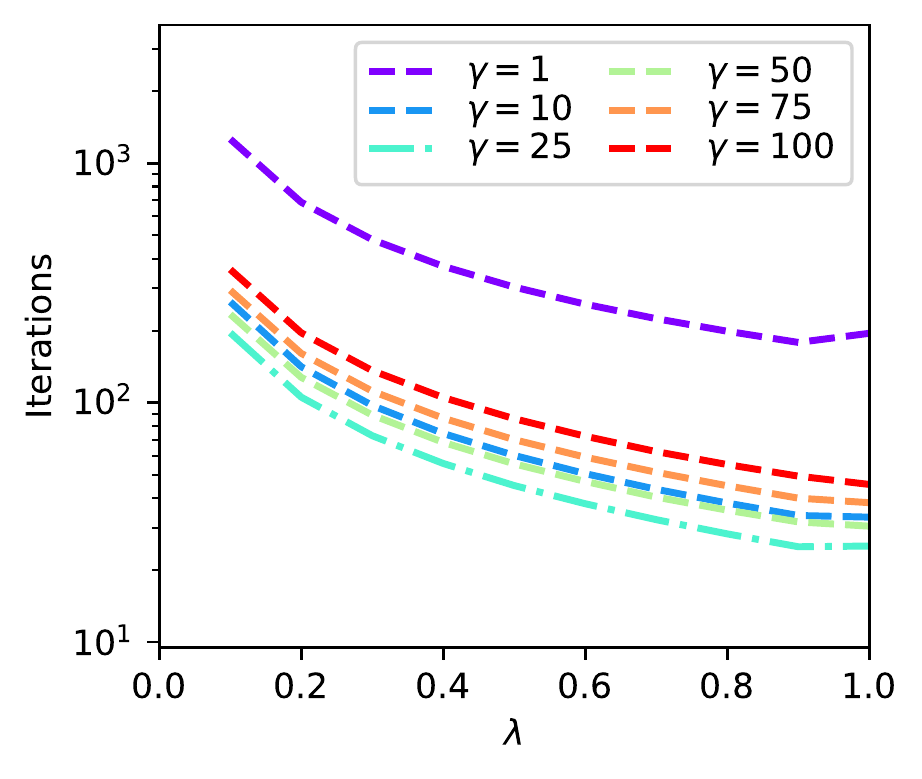}
     \subcaption{Malitsky--Tam}
  \end{subfigure}
  \begin{subfigure}[c]{0.44\textwidth}
     \includegraphics[width=\fw\textwidth]{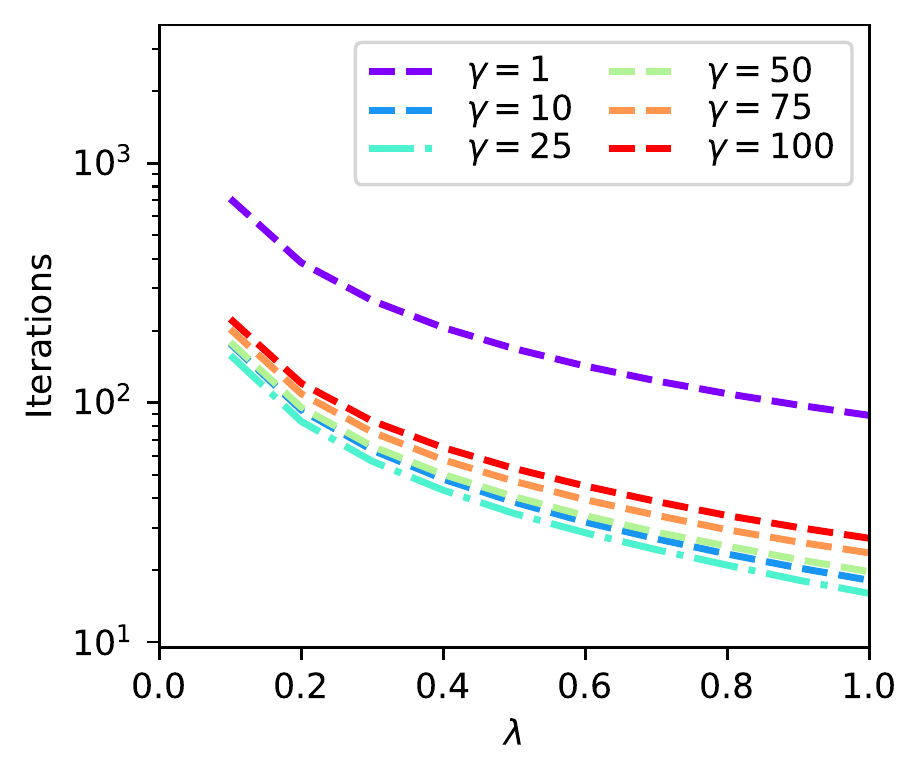}
     \subcaption{Ryu}
  \end{subfigure}
\caption{Performance of Standard-DR, Reduced-DR, Malitsky--Tam  and Ryu algorithms  for solving the generalized Heron problem in $\R^{100}$ with $r=3$. For each pair of parameters $(\gamma,\lambda)$, we represent the average number of iterations among $10$ problems and $10$ random~starting~points~each.}\label{fig:Heron_iter}
\end{figure}

\begin{table}[ht!]
\centering
\begin{tabular}{cccc}
\hline
Algorithm & $\gamma$ & $\lambda$ & Average iterations  \\
\hline
{\em Standard-DR} & $25$ & $1.2$ & $44.15$  \\
{\em Reduced-DR} & $25$ & $1.3$ & $13.41$ \\
{\em Malitsky--Tam} & $25$ & $0.9$ & $25.00$\\
{\em Ryu} & $25$ & $1.0$ & $15.96$ \\
\hline
\end{tabular}
\caption{Best choice of parameters and minimum averaged number of iterations, among $10$ problems and $10$ random~starting~points~each, required by Standard-DR, Reduced-DR, Malitsky--Tam and Ryu algorithms for solving the generalized Heron problem in $\R^{100}$ with $r=3$.}
\label{tbl:Heron}
\end{table}

Once the parameters had been tuned, we analyzed the effect of the dimension of the space ($n$), as well as the number of operators ($r$), on the comparison between all algorithms. For the first purpose, we fixed $r=3$ and generated $20$ problems in $\R^n$ for each $n\in\{100,200,\ldots,1000\}$. Then, for each problem we computed the average time, among $10$ random starting points, required by each algorithm to converge. Parameters $\gamma$ and $\lambda$ were chosen as in \Cref{tbl:Heron} according to the previous experiment. The results, shown in \Cref{fig:Heron_size}, confirm the consistent advantage of Reduced-DR and Ryu for all sizes. Indeed, these two algorithms were around 4 times faster than Standard-DR, whereas Malitsky--Tam was 2 times faster than Standard-DR.

For the second objective we repeated the experiment where now, for each number of operators $r\in\{3,4,\ldots,20\}$, we generated $20$ problems in $\R^{100}$. We did not consider Ryu splitting algorithm since it is only devised for three operators. We show the results in \Cref{fig:Heron_squares}, from which we deduce that the superiority of Reduced-DR and Malitsky--Tam over Pierra's standard reformulation is diminished as the number of operators increases. However, this drop is more drastic for the Makitsky--Tam algorithm. In fact,  while Reduced-DR is still always preferable to Standard-DR for all the considered values of $r$, Malitsky--Tam algorithm turns even slower than the classical approach when the number of operators is greater than~$6$.

\begin{figure}[!htbp]
\centering
  \hspace{0.02\textwidth}\begin{subfigure}[c]{0.46\textwidth}
     \includegraphics[width=\fw\textwidth]{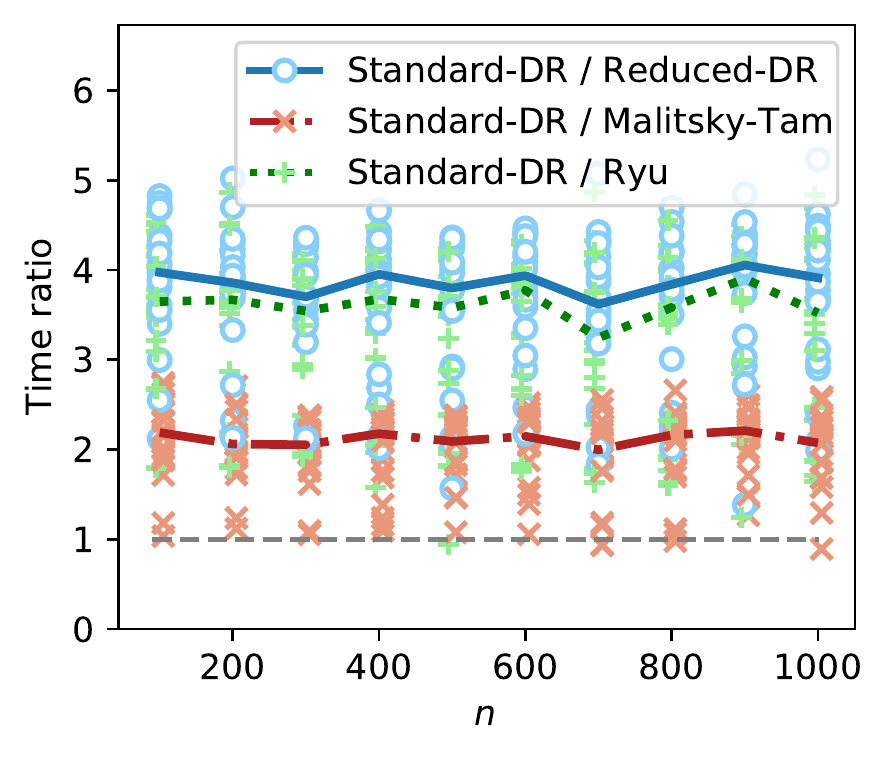}
     \subcaption{$n\in\{100,200,\ldots,1000\}$ and $r=3$.}\label{fig:Heron_size}
  \end{subfigure}\hspace{0.04\textwidth}
  \begin{subfigure}[c]{0.46\textwidth}
     \includegraphics[width=\fw\textwidth]{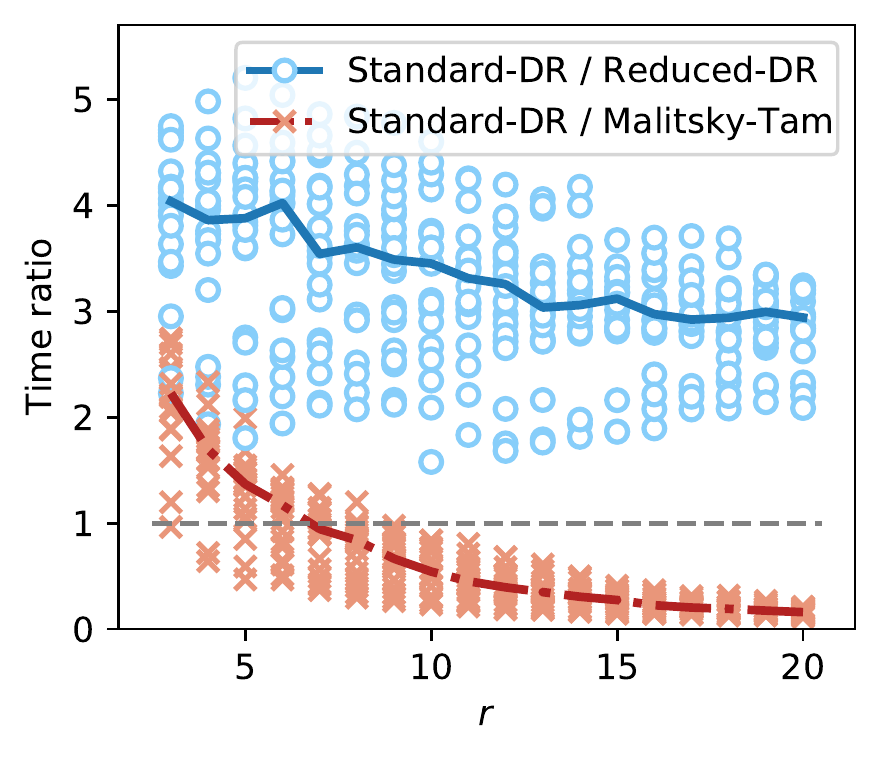}
     \subcaption{$n=100$ and $r\in\{3,4,\ldots,20\}$.}\label{fig:Heron_squares}
  \end{subfigure}
\caption{Comparison of the performance of Standard-DR, Reduced-DR, Malitsky--Tam and Ryu algorithms for solving $20$ instances of generalized Heron problem with $r$ sets in $\R^{n}$ for different values of $n$ and $r$. For each problem we represent the ratio between the average time required by each algorithm over Reduced-DR, among $10$ random starting points. The colored lines connect the median of the ratios while the dashed grey line represents ratios~equal~to~$1$.}\label{fig:Heron_analysis}
\end{figure}

\subsection{Sudoku puzzles}

In this section we analyze the potential of the product space reformulation with reduced dimension for nonconvex feasibility problems~(\Cref{p:npsf}). To this aim, we concern on Sudoku puzzles, which were first investigated by the Douglas--Rachford algorithm in~\cite{Elser}. Since then, other formulations as feasibility problems have been studied; see, e.g.,~\cite{aragon2014recents,aragon2018gc,aragon2020elser}. In this paper we consider the formulation with binary variables described in~{\cite[Section~6.2]{aragon2014recents}}, which we explain next.

Recall that a Sudoku puzzle is defined by a $9\times 9$ grid, composed by nine $3\times 3$ subgrids, where some of the cells are prescribed with some given values.  The objective is to fill the remaining cells so that each row, each column and each subgrid contains the digits from $1$ to $9$ exactly once. Possible solutions to a given Sudoku are encoded as a 3-dimensional multiarray $X\in\R^{9\times 9\times 9}$ with binary entries defined componentwise as
\begin{equation}\label{eq:Sud_enco}
X[i,j,k]=\left\{\begin{array}{cl}
1,& \text{if digit } k \text{ is assigned to the } (i,j)\text{th entry of the Sudoku},\\
0,& \text{otherwise;}
\end{array}\right.
\end{equation}
for $(i,j,k)\in I^3$ where $I:=\{1,2,\ldots,9\}$. Let $\C:=\{e_1,e_2,\ldots,e_9\}$ be the standard basis of $\R^9$, let $J\subseteq I^3$ be the set of indices for the prescribed entries of the Sudoku, and denote by $\ve M$ the vectorization, by columns, of a matrix $M$. Under encoding \eqref{eq:Sud_enco}, a solution to the Sudoku can be found by solving the feasibility problem

\begin{equation}\label{eq:formulation}
   \text{Find }X\in C_1\cap C_2\cap C_3\cap C_4\cap C_5 \subseteq \R^{9\times 9\times 9},
\end{equation}
where the constraint sets are defined by
\begin{equation*}\label{eq:constraints}
 	\begin{aligned}
 	C_1  &:= \left\{
 	X\in\R^{9\times 9\times 9}: X[i,:,k]\in\C, \forall i,k\in I
 	\right\}, \\
 	C_2  &:= \left\{
 	X\in\R^{9\times 9\times 9}: X[:,j,k]\in\C, \forall j,k\in I
 	\right\}, \\
 	C_3  &:= \left\{
 	X\in\R^{9\times 9\times 9}: X[i,j,:]\in\C, \forall i,j\in I
 	\right\}, \\
 	C_4  &:= \left\{
 	X\in\R^{9\times 9\times 9}: \begin{array}{c}\ve X[3i+1:3(i+1),3j+1:3(j+1),k]\in\C,\\ \forall i,j\in\{0,1,2\}, \forall k\in I \end{array}
 	\right\}, \\
 	C_5  &:= \left\{
 	X\in\R^{9\times 9\times 9}: X[i,j,k]=1, \forall (i,j,k)\in J
 	\right\}. \\
 	\end{aligned}
\end{equation*}
Observe that nonconvexity of problem \eqref{eq:formulation} arises from the combinatorial structure of $C_1$, $C_2$, $C_3$, $C_4\subseteq\{0,1\}^{9\times 9\times 9}$. Projections onto these sets can be computed by means of the projector mapping onto $\C$ (see~\cite[Remark~5.1]{aragon2014recents}). On the other hand, $C_5$ is an affine subspace of $\R^{9\times 9\times 9}$ whose projector can be readily computed component-wise as
\begin{equation*}
P_{C_5}(X)[i,j,k]=\left\{\begin{array}{ll}
1,& \text{if } (i,j,k)\in J,\\
X[i,j,k],& \text{otherwise.}
\end{array}\right.
\end{equation*}
 
In our experiment we considered the 95~hard puzzles from the library \texttt{top95}\footnote{\texttt{top95}: \url{http://magictour.free.fr/top95}}. For each puzzle, we run Standard-DR, Reduced-DR and Malitsky--Tam from $10$ random initial points. Parameter $\lambda$ was roughly tuned for good performance and it was fixed to $\lambda=1$ for Standard-DR and Reduced-DR and $\lambda=0.5$ for Malitsky--Tam. The algorithms were stopped when either they found a solution or when the CPU running time exceeded 5 minutes. A summary of the results can be found in \Cref{tbl:Sudoku}. While the success of all three algorithms is very similar,  the average CPU time and, specially, the proportion of wins are clearly favorable to Reduced-DR.

\begin{table}[ht!]
\centering
\begin{tabular}{cccc}
\hline
Algorithm & Solved & Wins & Time (median) \\
\hline
{\em Standard-DR} & $91.78\%$ & $18.94\%$ & $0.7423~s.$ \\
{\em Reduced-DR} & $91.78\%$ & $55.47\%$ & $0.5904~s.$ \\
{\em Malitsky--Tam} & $88.73\%$ & $15.05\%$ & $0.8836~s.$ \\
\hline
\end{tabular}
\caption{Results of the comparison between Standard-DR, Reduced-DR and Malitsky--Tam algorithm for solving 95 Sudoku problems from $10$ random starting points each. For each algorithm, we show the percentage of solved instances, the percentage of instances for which the algorithm was fastest, and the median of the CPU time required among the solved instances. Instances were labeled as unsolved after 5 minutes.}
\label{tbl:Sudoku}
\end{table}

In order to better visualize the results we turn to \emph{performance profiles} (see~\cite{DM02} and the modification proposed in~\cite{ISU16}), which are constructed as explained next.

\paragraph{Performance profiles}
Let $\mathcal{A}$ denote a set of algorithms to be tested on be a set of $N$ problems, denoted by $\mathcal{P}$, for multiple runs (starting points). Let $s_{a,p}$ denote the fraction of successful runs of algorithm $a\in\mathcal{A}$ on problem $p\in\mathcal{P}$ and let $t_{a,p}$ be the averaged time required to solve those successful runs. Compute $t^\star_p:=\min_{a\in\mathcal{A}} t_{a,p}$ for all $p\in\mathcal{P}$. Then, for any $\tau\geq 1$, define $R_a(\tau)$ as the set of problems for which algorithm $a$ was at most $\tau$ times slower than the best algorithm; that is, $R_a(\tau):=\{p\in\mathcal{P}, t_{a,p}\leq\tau t^\star_p\}$. The \emph{performance profile} function of algorithm $a$ is~given~by
\begin{equation*}
\begin{array}{rccl}
\rho_a: & [1,+\infty) & \longmapsto & [0,1]\\
& \tau &  \mapsto & \rho_a(\tau):=\frac{1}{N}\sum_{p\in R_a(\tau)} s_{a,p}.
\end{array}
\end{equation*}
The value $\rho_a(1)$ indicates the portion of runs for which $f$ was the fastest formulation. When $\tau\rightarrow+\infty$, then $\rho_a(\tau)$ gives the proportion of successful runs for formulation $f$.\\

Performance profiles of the results of Sudoku experiment are shown in \Cref{fig:Sudoku}, which confirm the conclusions drawn from \Cref{tbl:Sudoku}. Furthermore, we can now asses that Reduced-DR becomes consistently superior since its performance profile is mostly above the one of the remaining two algorithms.

\begin{figure}[!htbp]
\centering
  	\begin{subfigure}[c]{0.52\textwidth}
     \includegraphics[width=\textwidth]{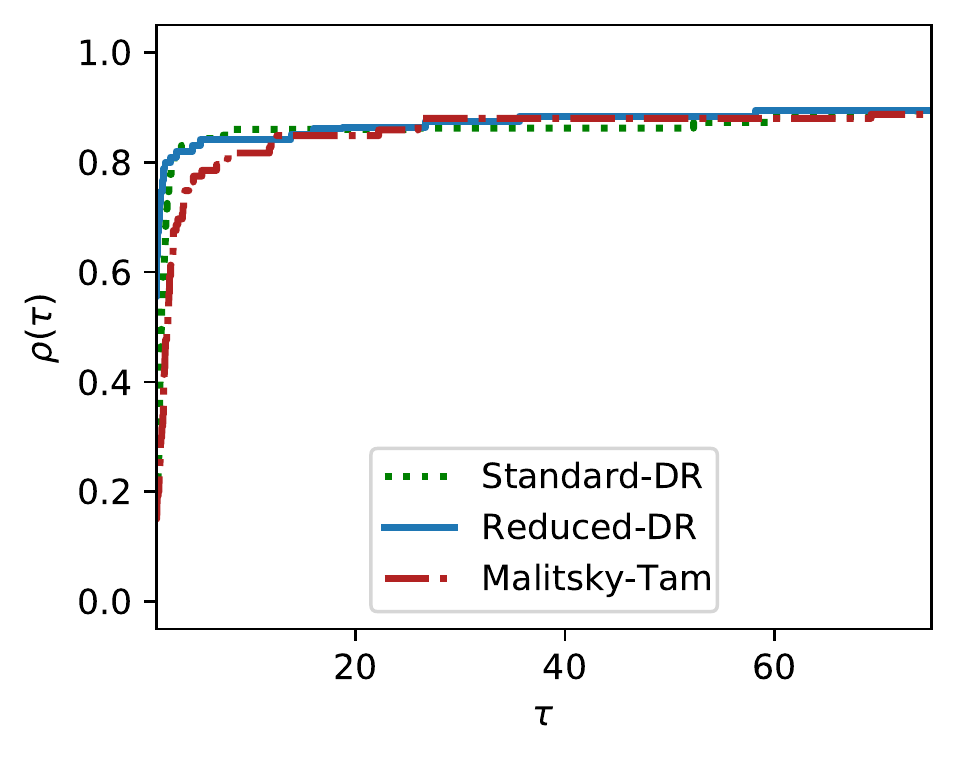}
     \end{subfigure}
     \begin{subfigure}[c]{0.36\textwidth}
     \includegraphics[width=\textwidth]{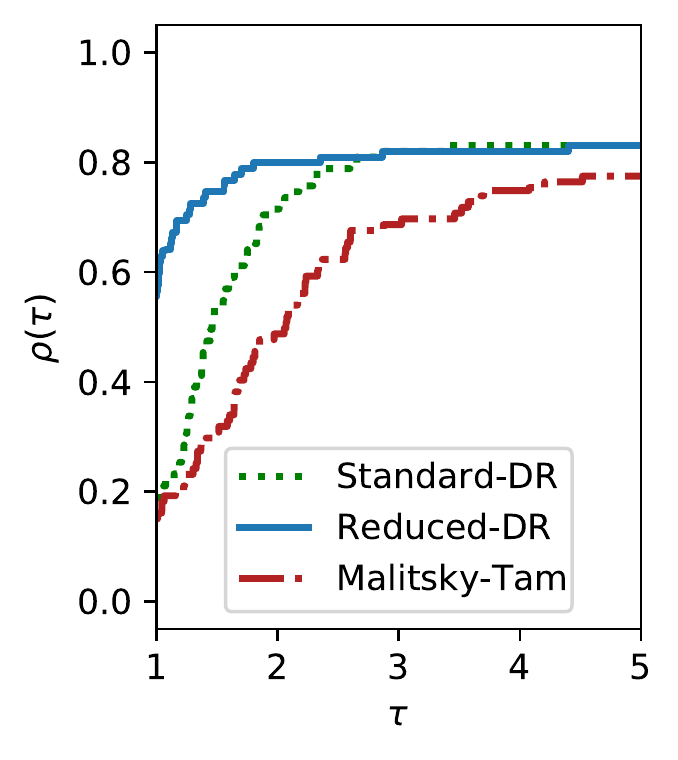}
     \end{subfigure}
\caption{Performance profiles comparing Standard-DR, Reduced-DR  and Malitsky--Tam algorithm for solving 95 Sudoku problems (left). For each problem, 10 random starting points were considered. Instances were labeled as unsolved after 5 minutes of CPU running time. For the sake of clarity we focus the view of the performance profiles to the values of  $\tau\in[1,5]$ (right).}\label{fig:Sudoku}
\end{figure}

We would like to conclude with the following comment regarding the implementation of splitting algorithms on \eqref{eq:formulation}. 

\begin{remark}[On the order of the sets]
Observe that Pierra's classical reformulation in \Cref{p:psf}, and thus Standard-DR, is completely symmetric on the order of the sets $C_1,\ldots,C_5$. However, this is not the case for the reformulation in \Cref{p:npsf}, where one has to decide which of the sets will be merged to the diagonal to construct the set $\bs K$ in \eqref{eq:npsf_K}. In our test, we followed the arrangement in \Cref{p:npsf}, that is, 
\begin{align*}
\bs B:= & C_1\times C_2\times C_3\times C_4 \subseteq (\R^{9\times 9\times 9})^{4},\\[0.2em]
\bs K:= & \{(x,x,x,x)\in(\R^{9\times 9\times 9})^{4}: x\in C_5\}\subseteq (\R^{9\times 9\times 9})^{4}.
\end{align*}
Note that this makes the constrained diagonal set $\bs K$ to be an affine subspace. Due to the nonconvexity of the problem, the reformulation chosen may be crucial for the success of the algorithm. For example, we tested all the remaining combinations, for which Reduced-DR rarely found a solution on the considered problems within the first 5 minutes of~running~time.
\end{remark}

{\small
\paragraph{\small Acknowledgements}
The author was partially supported by the Ministry of Science, Innovation and Universities of Spain and the European Regional Development Fund (ERDF) of the European Commission (PGC2018-097960-B-C22), and by the Generalitat Valenciana (AICO/2021/165).

}

\end{document}